\documentclass[11pt,reqno]{amsart}
\usepackage{amsmath,amsfonts,amssymb,amsthm,amscd,comment,euscript}
\usepackage[all]{xy}
\usepackage{graphicx}
\usepackage{mathptmx}
\usepackage{enumerate} 
\usepackage[colorlinks=true]{hyperref} 
\usepackage[usenames,dvipsnames]{xcolor}
\usepackage{enumitem}

\xymatrixcolsep{1.9pc}                          
\xymatrixrowsep{1.9pc}
\newdir{ >}{{}*!/-5\pt/\dir{>}}                  

 \calclayout

\swapnumbers

\theoremstyle{plain}
\newtheorem{lem}{Lemma}[section]
\newtheorem{cor}[lem]{Corollary}

\newtheorem{thm}[lem]{Theorem}

\theoremstyle{definition}
\newtheorem{ex}[lem]{Example}

\newtheorem{rem}[lem]{Remark}
\newtheorem{dfn}[lem]{Definition}

\newtheorem{axioms}[lem]{Axioms}
\newtheorem{evaluationfunctors}[lem]{Evaluation Functors}

\renewcommand{\phi}{\varphi}
\renewcommand{\leq}{\leqslant}
\renewcommand{\geq}{\geqslant}
\renewcommand{\epsilon}{\varepsilon}

\renewcommand{\kappa}{\varkappa}

\DeclareMathOperator{\spec}{Spec}
 \DeclareMathOperator{\cyl}{cyl}
\DeclareMathOperator{\proj}{proj}

 \DeclareMathOperator{\mot}{mot}

\DeclareMathOperator{\Hom}{Hom} 
 \DeclareMathOperator{\id}{id}

 \DeclareMathOperator{\colim}{colim}

\DeclareMathOperator{\pr}{pr}
 \DeclareMathOperator{\nis}{\mathsf{nis}}

 \DeclareMathOperator{\Ob}{Ob}

\newcommand{\one}{\bf{1}}
\newcommand{\MGL}{\bf{MGL}}
\newcommand{\MSp}{\bf{MSp}}
\newcommand{\MSL}{\bf{MSL}}

\newcommand{\MZ}{\bf{M}\mathbb{Z}}
\newcommand{\MWMZ}{\widetilde{\bf{M}}\mathbb{Z}}

\newcommand{\lra}[1]{\bl{#1}\longrightarrow\relax}
\newcommand{\bl}[1]{\buildrel #1\over}
\newcommand{\cc}{\mathcal}
\newcommand{\bb}{\mathbb}

\newcommand{\op}{{\textrm{\rm op}}}

\newcommand{\wh}{\widehat}
\newcommand{\wt}{\widetilde}

\newcommand{\gmp}{\bb G}
\newcommand{\gmpn}{\bb G^{\wedge n}}
\newcommand{\uhom}{\underline{\Hom}}

\newcommand{\M}{\mathcal{M}}

\newcommand{\SH}{\mathbf{SH}}
\newcommand{\Sp}{\mathbf{Sp}}
\newcommand{\eff}{\mathsf{eff}}
\newcommand{\veff}{\mathsf{veff}}
\newcommand{\fr}{\mathsf{fr}}
\newcommand{\Ho}{\mathbf{H}}
\newcommand{\veffr}{\mathsf{veffr}}
\newcommand{\pt}{\mathsf{pt}}
\newcommand{\ev}{\mathsf{ev}}

\begin{document}

\footskip30pt

\baselineskip=1.1\baselineskip

\title{Framed motivic $\Gamma$-spaces}
\author{Grigory Garkusha}
\address{Department of Mathematics, Swansea University, Fabian Way, Swansea SA1 8EN, UK}
\email{g.garkusha@swansea.ac.uk}

\author{Ivan Panin}
\address{St. Petersburg Branch of V. A. Steklov Mathematical Institute,
Fontanka 27, 191023 St. Petersburg, Russia \& Department of Mathematics, University of Oslo, P.O.~Box 1053 Blindern, 0316 Oslo, Norway}
\email{paniniv@gmail.com}

\author{Paul Arne {\O}stv{\ae}r}
\address{Department of Mathematics, University of Oslo, P.O.~Box 1053 Blindern, 0316 Oslo, Norway}
\email{paularne@math.uio.no}

\begin{abstract}
We combine several mini miracles to achieve an elementary understanding of infinite loop spaces
and very effective spectra in the algebro-geometric setting of motivic homotopy theory.
Our approach combines $\Gamma$-spaces and Voevodsky's framed
correspondences into the concept of framed motivic $\Gamma$-spaces;
these are continuous or enriched functors of two variables that take values in 
framed motivic spaces.
We craft proofs of our main results by imposing further axioms on framed motivic $\Gamma$-spaces such as
a Segal condition for simplicial Nisnevich sheaves,
cancellation,
${\bb A}^{1}$- and $\sigma$-invariance,
Nisnevich excision,
Suslin contractibility,
and grouplikeness.
This adds to the discussion in the literature on coexisting points of view on the ${\bb A}^{1}$-homotopy theory of algebraic varieties.
\end{abstract}

\dedicatory{In memory of Vladimir Voevodsky}


\maketitle
\thispagestyle{empty}
\pagestyle{plain}

\section{Introduction}
\label{section:introduction}
The category $\Gamma$ of correspondences or multivalued functions on finite sets is of fundamental importance in topology \cite{S}.
Following Boardman and Vogt \cite{BV},
Segal's work on $\Gamma$-spaces give convenient models for $E_{\infty}$ spaces
--- spaces with multiplications that are unital, associative, and commutative up to higher coherent homotopies ---
and for infinite loop spaces.
Segal applied his ideas to prove the celebrated Barratt-Priddy-Quillen theorem identifying the group completion of the disjoint union $\sqcup_{n} B\Sigma_{n}$ of
classifying spaces of symmetric groups with the infinite loop space of  $\bb S$ ---
the topological sphere.
Soon afterwards,
Bousfield and Friedlander carried out their homotopical identification of connective spectra and $\Gamma$-spaces
--- an early striking success in the development of stable homotopy theory \cite{BFgamma}.
Moreover,
$\Gamma$-spaces have the advantage that they are simple to define,
as well as being intrinsically tied to $K$-theory and topological Hochschild homology \cite{DGM}.
\vspace{0.1in}

In this paper we introduce the concept of framed motivic $\Gamma$-spaces together with a few axioms.
The main purpose of our set-up is to advance our practical understanding of infinite loop spaces along with new
viewpoints on connective and very effective spectra
in the algebro-geometric setting of ${\bb A}^{1}$-homotopy theory \cite{MV}, \cite{VoeICM}.
Voevodsky envisioned this new direction of development in his work on framed correspondences in motivic homotopy theory \cite{Voe2}.
\vspace{0.1in}

Working over a field $k$,
our approach combines Segal's category $\Gamma$ with Voevodsky's symmetric monoidal category $\mathrm{Sm}/k_{+}$ of framed correspondences of level zero \cite{Voe2}
--- a slight enlargement of $\mathrm{Sm}/k$, the category of smooth separated schemes of finite type over $\spec(k)$.
\vspace{0.1in}

Recall from \cite{GP1} that a framed motivic space is a pointed simplicial  Nisnevich
sheaf on the category of framed correspondences $\mathrm{Fr}_{+}(k)$. As noted in \S\ref{section:preliminaries},
$\mathrm{Sm}/k_{+}$, the opposite category $\Gamma^{\op}$ of pointed finite sets and the category of framed motivic
spaces $\cc M^{\fr}$ are enriched in the closed symmetric
monoidal category of pointed motivic spaces $\M$ \cite{DLORV}.
With respect to the said enrichments,
we shall consider ``continuous functors in two variables" for the monoidal product of $\Gamma^{\op}$ and $\mathrm{Sm}/k_{+}$
taking values in framed motivic spaces
\begin{equation}
\label{equation:mgs}
\cc X
\colon
\Gamma^{\op}
\boxtimes
\mathrm{Sm}/k_{+}
\longrightarrow
\M^{\fr}
\end{equation}
and call them {\it framed motivic $\Gamma$-spaces\/}
(see Defintion~\ref{dfn:gammasp}).
We should note that there is a canonically induced faithful functor
\begin{equation}
\label{equation:MfrtoM}
\M^{\fr}
\longrightarrow
\M
\end{equation}
obtained from the composite
\begin{equation}
\label{equation:SmtoFr}
\mathrm{Sm}/k
\longrightarrow
\mathrm{Sm}/k_{+}
\longrightarrow
\mathrm{Fr}_{+}(k).
\end{equation}
The definition of framed correspondences
invented in \cite{Voe2} uses an
algebro-geometric analogue of a framing on the stable normal bundle of a manifold.
We shall review additional background on \eqref{equation:mgs}, \eqref{equation:MfrtoM},
and \eqref{equation:SmtoFr} in \S\ref{section:preliminaries}.
\vspace{0.1in}

In our quest to carry over Segal's programme for $\Gamma$-spaces to ${\bb A}^{1}$-homotopy
theory we begin by formulating some homotopical axioms for framed motivic $\Gamma$-spaces.
These axioms concern both of the ``variables" $\Gamma^{\op}$ and $\mathrm{Sm}/k_{+}$ in \eqref{equation:mgs}.
Informally speaking,
the pointed finite sets accounts for the $S^{1}$-suspension whereas the framed correspondences
accounts for the $\bb G_{m}$-suspension in stable motivic homotopy theory.
We may and will view $\Gamma^{\op}$ as the full subcategory of pointed finite sets with
objects $n_{+}=\{0,\dots,n\}$ pointed at $0$ for every integer $n\geq 0$.
\vspace{0.1in}

Every $\Gamma$-space gives rise to a simplicial functor and hence an associated $S^{1}$-spectrum; for details, see \cite[Chapter 2]{DGM}.
Similarly in the motivic setting,
see \eqref{equation:S1evaluation},
we show that every $U\in \mathrm{Sm}/k_{+}$ and $\cc X$ as in \eqref{equation:mgs} give rise to a presheaf of $S^{1}$-spectra $\cc X(\bb S,U)$.
We refer to \cite{JardineLocal} for a comprehensive introduction to the homotopical algebra of such presheaves.
In Axioms \ref{axioms:motivicgammaspaces} below we employ the notions of local equivalences for simplicial presheaves \cite[Chapter 4]{JardineLocal} and stable local equivalences
for presheaves of $S^{1}$-spectra \cite[Chapter 10]{JardineLocal}.
\vspace{0.1in}

For $n\geq 0$ and every finitely generated field extension $K/k$ we write $\wh{\Delta}^{n}_{K/k}$ for the semilocalization of the standard algebraic $n$-simplex
$$
\Delta^n_K
=
\spec(K[x_0,\ldots,x_n]/(x_0+\cdots+x_n-1))
$$
with closed points the vertices $v_{0},\dots,v_{n}\in \Delta^n_K$ --- see \cite[\S3]{VoeICM} for the colimit preserving realization functor from simplicial sets to Nisnevich sheaves.
We recall that $v_{i}$ is the closed subscheme of $\Delta^n_K$ defined by $x_{j}=0$ for $j\neq i$, $0\leq i\leq n$.
Following \cite[\S2]{Suslin} we write $\wh{\Delta}^{\bullet}_{K/k}$ for the corresponding cosimplicial semilocal scheme.
\vspace{0.1in}

We are ready to introduce the main objects of study in this paper.
\begin{axioms}
\label{axioms:motivicgammaspaces}
A framed motivic $\Gamma$-space $\cc X$  is called {\it special} if (1)-(5) holds:
\begin{enumerate}
\item
We have $\cc X(0_{+},U)=\ast=\cc X(n_{+},\emptyset)$ for all $n\geq 0$ and $U\in \mathrm{Sm}/k_{+}$,
while for all $n\geq 1$ and nonempty $U\in \mathrm{Sm}/k_{+}$ the naturally induced morphism
$$
\cc X(n_{+},U)
\longrightarrow
\cc X(1_{+},U)
\times
\bl n
\cdots
\times
\cc X(1_{+},U)
$$
is a local equivalence of pointed motivic spaces.
\vspace{0.05in}

\item
For all $n\geq 0$ and $U\in \mathrm{Sm}/k_{+}$ the framed presheaf of stable homotopy groups
$$
V
\longmapsto
\pi^{s}_{n}\cc X(\bb S,U)(V)
$$
is ${\bb A}^{1}$-invariant, radditive and $\sigma$-stable (see Remark~\ref{somedefs}).
\vspace{0.05in}

\item (Cancellation)
Let $\gmp$ denote the cone of the $1$-section $\mathrm{Spec}(k)\longrightarrow\bb G_{m} $ in $\Delta^{\op}\mathrm{Sm}/k_{+}$.
For all $n\geq 0$ and $U\in \mathrm{Sm}/k_{+}$ there is a canonical stable local equivalence
$$
\cc X(\bb S,\gmpn\times U)
\longrightarrow
\underline{\Hom}(\gmp,\cc X(\bb S,\bb G^{\wedge n+1}\times U)).
$$

\item (${\bb A}^{1}$-invariance)
For all $U\in \mathrm{Sm}/k_{+}$ there is a naturally induced stable local equivalence
$$
\cc X(\bb S,U\times\bb A^1)\longrightarrow\cc X(\bb S,U).
$$

\item (Nisnevich excision)
For every elementary Nisnevich square in $\mathrm{Sm}/k$
$$
\xymatrix{
U'\ar[r]\ar[d]&X'\ar[d]\\
U\ar[r]&X }
$$
there is a homotopy cartesian square in the stable local model structure:
$$
\xymatrix{
\cc X(\bb S,U')\ar[r]\ar[d]&\cc X(\bb S,X')\ar[d]\\
\cc X(\bb S,U)\ar[r]&\cc X(\bb S,X)
}
$$

Moreover, a special framed motivic $\Gamma$-space $\cc X$ is called {\it very effective\/} if (6) holds and {\it very special\/} if (7) holds.
\vspace{0.05in}

\item(Suslin contractibility)
For all $U\in \mathrm{Sm}/k_{+}$ and any finitely generated field extension $K/k$,
the geometric realization of the simplicial $S^1$-spectrum
$$
\cc X(\bb S,\bb G\times U)(\wh{\Delta}^{\bullet}_{K/k})
$$ is contractible.
\vspace{0.05in}

\item(Grouplikeness)
For all $U\in \mathrm{Sm}/k_{+}$ the Nisnevich sheaf $\pi^{\nis}_{0}\cc X(1_{+},U)$ associated to the presheaf
$$
V
\longmapsto
\pi_{0}\cc X(1_{+},U)(V)
$$
of connected components on $\mathrm{Sm}/k$ takes values in abelian groups.
\end{enumerate}
\end{axioms}

\begin{rem}\label{somedefs}
The reader will recognize axioms (1) and (7) as sheaf versions of special and very special Segal $\Gamma$-spaces,
respectively \cite{BFgamma,S}.
Axiom (2) makes use of the assumption that $\cc X$ is a framed motivic $\Gamma$-space.
A framed presheaf $\cc F$ is {\it $\sigma$-stable\/} if $\cc F(\sigma_{V})=\id_{\cc F(V)}$ for all $V\in\mathrm{Sm}/k$.
Here the level $1$ explicit framed correspondence
$(\{0\}\times V,{\bb A}^{1}\times V,\pr_{\bb A^{1}},\pr_{V})\in\mathrm{Fr}_{1}(V,V)$ defines a map
$\sigma_{V}\colon V\longrightarrow V$ in $\mathrm{Fr}_{+}(k)$;
see \cite[\S2]{GP1}. $\cc F$ is {\it radditive\/} if $\cc F(\emptyset)=*$ and $\cc F(X_1\sqcup X_2)=\cc F(X_1)\times\cc F(X_2)$
for all $X_1,X_2\in\mathrm{Sm}/k$. In (3),
$\gmp$ is a simplicial object in $\mathrm{Sm}/k_{+}$ with smash product $\bb G^{\wedge n}$
formed in $\Delta^{\op}\mathrm{Sm}/k_{+}$ \cite[Notation 8.1]{GP1}.
Axioms (2), (3), (4), and (5) are concerned with presheaves of $S^{1}$-spectra as in \cite[Part IV]{JardineLocal}.
Axiom (6) traces back to Suslin's work on rationally contractible presheaves in \cite{Suslin};
see also \cite{BF} and \cite{GP5}.
\end{rem}

\begin{ex}
\label{ex:iutrhgvkubk}
An example of a quintessential special framed motivic $\Gamma$-space is given by
$$
(n_{+},U)
\in
\Gamma^{\op}\boxtimes \mathrm{Sm}/k_{+}
\longmapsto
C_{*}\mathrm{Fr}(-,n_{+}\otimes U)
\in
\M^{\fr}.
$$
Here $\mathrm{Fr}$ refers to stable framed correspondences and $C_{*}\mathrm{Fr}(-,X')$ to the simplicial framed functor
$X\longmapsto \mathrm{Fr}(X\times {\Delta}^{\bullet}_{k},X')$
---
see \cite{GP1} and \cite{Voe2}.
By $K\otimes U$,
where $K\in\Gamma^{\op}$ and $U\in\mathrm{Sm}/k$,
we mean the coproduct of copies of $U$ indexed by the non-based elements in $K$.
\end{ex}

The evaluation functor in \eqref{equation:S1Gevaluationfr} associates to every framed
motivic $\Gamma$-space $\cc X$ an object in the category of framed motivic spectra
in the sense of~\cite[Definition 2.1]{GP5}
$$
\cc X_{S^{1},\bb G}
\in
\mathbf{Sp}^{\fr}_{S^1,\bb G}(k).
$$

Recall that the triangulated category of framed bispectra $\SH^{\fr}_{\nis}(k)$ whose objects are those of
$\mathbf{Sp}^{\fr}_{S^1,\bb G}(k)$ is equivalent to the stable motivic homotopy category
$\SH(k)$ via the identity with quasi-inverse the big framed motive functor \cite[Theorem 2.2]{GP5}.
The big framed motive functor is closely related to Example \ref{ex:iutrhgvkubk}
--- for details we refer to \cite[Section 12]{GP1}.
\vspace{0.1in}

For the purposes of this paper it is not necessary to discuss model structures on framed motivic $\Gamma$-spaces.
Our next definition is inspired by Segal's homotopy category of $\Gamma$-spaces \cite{S}.

\begin{dfn}
\label{fmgshc}
The {\it homotopy category of framed motivic $\Gamma$-spaces\/}
$$
\Ho_{\Gamma\M}^{\fr}(k)
$$
has objects special framed motivic $\Gamma$-spaces and with morphisms given by
$$
\Ho_{\Gamma\M}^{\fr}(k)(\cc X,\cc Y):=\SH^{\fr}_{\nis}(k)(\cc X_{S^{1},\bb G},\cc Y_{S^{1},\bb G}).
$$
\end{dfn}
\vspace{0.1in}

In \S\ref{fmg} we discuss how $\Ho_{\Gamma\M}^{\fr}(k)$ relates to the unstable
pointed motivic homotopy category $\Ho(k)$ and to connective motivic spectra $\SH(k)_{\geq 0}$
via the commutative
--- up to equivalence of functors ---
diagram of adjunctions:
\begin{equation}
\label{equation:jeytrfgvjc}
\xymatrix{
\Ho(k)
\ar@/^/[rr]^-{\Sigma^{\infty}_{S^{1},\bb G}} \ar@/_/[dr]_-{C_{\ast}\mathrm{Fr}}
&&
\SH(k)_{\geq 0} \ar@/^/[dl]^-{\Gamma\bb M_{\fr}} \ar@/^/[ll]^-{\Omega^{\infty}_{S^{1},\bb G}}
\\
& \Ho_{\Gamma\M}^{\fr}(k) \ar@/^/[ur] \ar@/_/[ul]
}
\end{equation}
Here $\cc X\in \Ho_{\Gamma\M}^{\fr}(k)$ is mapped to its {\it underlying motivic space}
$\cc X (1_{+},\pt)\in \Ho(k)$ and to its {\it framed motivic spectrum}
$\cc X_{S^{1},\bb G}\in\SH(k)_{\geq 0}$ under the equivalence between $\SH^{\fr}_{\nis}(k)$ and $\SH(k)$ in \cite{GP5}.
We refer to Remark \ref{sergachev} for the definition of $\Gamma\bb M_{\fr}$
--- a version of the big framed motive functor introduced in \cite[Section 12]{GP1}.

\begin{thm}
\label{theorem:firstmain}
For every infinite perfect field $k$ there is an equivalence of categories
\begin{equation}
\label{equation:firstmainequivalence}
\Ho_{\Gamma\M}^{\fr}(k)
\overset{\simeq}{\longrightarrow}
\SH(k)_{\geq 0}
\end{equation}
$$
\cc X
\longmapsto
\cc X_{S^{1},\bb G}.
$$
Its quasi-inverse $\SH(k)_{\geq 0}\overset{\simeq}{\longrightarrow}\Ho_{\Gamma\M}^{\fr}(k)$
takes $\cc E\in\SH(k)_{\geq 0}$ to an explicitly constructed framed motivic $\Gamma$-space
$\Gamma\bb M^{\cc E}_{\fr}\in\Ho_{\Gamma\M}^{\fr}(k)$.
\end{thm}

Let $\SH^{\veff}(k)$ be the full subcategory of $\SH(k)$ that is generated under
homotopy colimits and extensions by motivic ${\bb P}^{1}$-suspension spectra of smooth schemes.
This category is of interest since it gives rise to the very effective slice filtration introduced in \cite{SO}.
We note $\SH^{\veff}(k)$ is contained in the triangulated category $\SH(k)_{\geq 0}$
--- generated under homotopy colimits and extensions by motivic ${\bb P}^{1}$-suspension spectra $\Sigma^{p,q}U_{+}$,
where $p\geq q$, and $U\in\mathrm{Sm}/k$.
\vspace{0.1in}

We shall study $\SH^{\veff}(k)$ from the point of view of framed motivic $\Gamma$-spaces.

\begin{dfn}
The {\it homotopy category of very effective framed motivic $\Gamma$-spaces\/}
$$
\Ho_{\Gamma\M}^{\veffr}(k)
$$
is the full subcategory of $\Ho_{\Gamma\M}^{\fr}(k)$ comprised of very effective special framed motivic $\Gamma$-spaces.
\end{dfn}

We show that Axiom (6) on Suslin contractibility of special framed motivic $\Gamma$-spaces captures
precisely the difference between $\SH^{\veff}(k)$ and $\SH(k)_{\geq 0}$.

\begin{thm}
\label{theorem:secondmain}
For every infinite perfect field $k$ there is an equivalence of categories:
\begin{equation}
\label{equation:secondmainequivalence}
\Ho_{\Gamma\M}^{\veffr}(k)
\overset{\simeq}{\longrightarrow}
\SH^{\veff}(k)
\end{equation}
$$
\cc X
\longmapsto
\cc X_{S^{1},\bb G}
$$
\end{thm}

Finally,
we employ Axiom (7) in our recognition principle for infinite motivic loop spaces.
\begin{thm}
\label{theorem:thirdmain}
For every infinite perfect field $k$ and every $\cc E\in \SH(k)$ there exists a
very special framed motivic $\Gamma$-space $\Gamma\bb M_{\fr}^{\cc E}$
and a local equivalence of pointed motivic spaces:
\begin{equation}
\label{equation:thirdmainequivalence}
\Gamma\bb M_{\fr}^{\cc E}(1_{+},\pt)\simeq
\Omega^{\infty}_{S^{1}}\Omega^{\infty}_{\bb G}\cc E
\end{equation}
Moreover,
if $\cc X$ is a very special framed motivic $\Gamma$-space then $\cc X(1_{+},\pt)$ is an infinite motivic loop space.
\end{thm}

\subsection*{Guide to the paper.}
For the convenience of the reader we begin \S\ref{section:preliminaries} by reviewing background on enriched categories,
with the aim of introducing framed motivic $\Gamma$-spaces.
As prime examples we discuss the motivic sphere spectrum $\one$,
algebraic cobordism $\MGL$,
motivic cohomology $\MZ$,
and Milnor-Witt motivic cohomology $\MWMZ$.
Our main results,
Theorems \ref{theorem:firstmain}, \ref{theorem:secondmain}, and \ref{theorem:thirdmain} are shown in \S\ref{fmg}.
Finally,
in \S\ref{hdtgtfe} we record some novel homotopical properties of framed motivic $\Gamma$-spaces.

\subsection*{Notation.}
Throughout the paper we employ the following notation.
\vspace{0.08in}

\begin{tabular}{l|l}
$k$, $\pt$ & infinite perfect field of exponential characteristic $e$, $\mathrm{Spec}(k)$\\
$\mathrm{Sm}/k$ & smooth separated schemes of finite type \\
$\mathrm{Sm}/k_{+}$ & framed correspondences of level zero\\
$\mathrm{Shv}_{\bullet}(\mathrm{Sm}/k)$ & closed symmetric monoidal category of pointed Nisnevich sheaves \\
$\M=\Delta^{\op}\mathrm{Shv}_{\bullet}(\mathrm{Sm}/k)$ & pointed motivic spaces, a.k.a.~pointed simplicial  Nisnevich sheaves \\
$\mathrm{Fr}_{+}(k)$ & the category of framed correspondences \\
$\mathrm{Pre}^{\fr}(k)$ & framed presheaves, a.k.a. presheaves of sets on $\mathrm{Fr}_{+}(k)$ \\
$i: \mathrm{Sm}/k\to\mathrm{Fr}_{+}(k)$ & the composite functor $\mathrm{Sm}/k\to\mathrm{Sm}/k_{+}\to\mathrm{Fr}_{+}(k)$ \\
$S^{s,t}$, $\Omega^{s,t}$, $\Sigma^{s,t}$ & motivic $(s,t)$-sphere, loop space, and suspension  \\
${\bf S}_{\bullet}$ & pointed simplicial sets
\end{tabular}
\vspace{0.05in}
\noindent

Our standard convention for motivic spheres is that $S^{2,1}\simeq\mathbb{P}^{1}\simeq T$
and $S^{1,1}\simeq\mathbb{A}^{1}\smallsetminus \{0\}$ as in \cite{MV}.

\subsection*{Relations to other works.}
Our approach in this paper is a homage to Segal's work on categories and cohomology theories \cite{S}.
Along the same line we use minimal machinery to achieve concrete models for infinite motivic loop spaces and motivic spectra with prescribed properties.
Based on Voevodsky's notes~\cite{Voe2},
the machinery of framed motives is developed in~\cite{GP1}.
As an application,
explicit computations of infinite motivic loop spaces are given as follows: $\Omega_{\bb P^1}^{\infty}\Sigma_{\bb P^1}^{\infty}A$,
$A\in\cc M$,
is locally equivalent to the space $C_{\ast}\mathrm{Fr}(A^c)^{\mathrm{gp}}$ (`gp' for group completion),
where $A^c$ is a projective cofibrant replacement of $A$
--- see~\cite[Section~10]{GP1}.
Based on~\cite{AGP,GP1,GP4,GNP},
a motivic recognition principle for infinite motivic loop spaces using the language of infinity categories is given in~\cite{EHKSY}.

\subsection*{Acknowledgments.}
The authors gratefully acknowledge support by the RCN Frontier Research Group Project no.~250399 ``Motivic Hopf Equations.''
Some work on this paper took place at the Institut Mittag-Leffler in Djursholm and the Hausdorff Research Institute for Mathematics in Bonn;
we thank both institutions for providing excellent working conditions, hospitality, and support.
{\O}stv{\ae}r was partially supported by Friedrich Wilhelm Bessel Research Award from the Humboldt Foundation,
Nelder Visiting Fellowship from Imperial College London,
Professor Ingerid Dal and sister Ulrikke Greve Dals prize for excellent research in the humanities,
and a Guest Professorship under the auspices of The Radbound Excellence Initiative.

\section{Framed motivic $\Gamma$-spaces}
\label{section:preliminaries}

We refer to \cite{Bl} and \cite{DRO2} for the projective motivic model structure on the closed symmetric
monoidal category of pointed motivic spaces $\M$.
This model structure is combinatorial, proper, simplicial, symmetric monoidal, and weakly finitely generated.
Let $\Delta[\bullet]$ be the standard cosimplicial simpicial set $n\longmapsto\Delta[n]$. If there is no
likelihood of confusion, we sometimes regard it as a cosimplicial smooth scheme, where each
$\Delta[n]$ is regarded as the disjoint union $\bigsqcup_{\Delta[n]}\pt$.
The simplicial function object between pointed motivic spaces $A$ and $B$ is given by
$$
{\bf S}_{\bullet}(A,B)=\Hom_{\M}(A\wedge\Delta[\bullet]_{+},B)
=\Hom_{\M}(A,B(\Delta[\bullet]\times-)).
$$
For every $U\in\mathrm{Sm}/k$ the Yoneda lemma identifies ${\bf S}_{\bullet}(U_{+},A)$ with the pointed simplicial set of sections $A(U)$.
\vspace{0.1in}

Recall $A\in\M$ is {\it finitely presentable\/} if the functor $\Hom_{\M}(A,-)$ preserves directed colimits.
For example,
the representable pointed motivic space $U_{+}$ is finitely presentable for every $k$-smooth scheme $U\in \mathrm{Sm}/k$.
\vspace{0.1in}

A collection $\cc C$ of finitely presentable pointed motivic spaces can be enriched in  $\M$ by means of the $\M$-enriched Hom-functor
\begin{multline}\label{tutu}
[A,B](X):=
\underline{\Hom}_{\M}(A,B)(X)
={\bf S}_{\bullet}(A\wedge X_{+},B)=\\
=\Hom_{\M}(A\wedge\Delta[\bullet]_{+},B(X\times-))
=\Hom_{\M}(A,B(X\times\Delta[\bullet]\times-)),
\quad A,B\in\cc C, X\in \mathrm{Sm}/k.
\end{multline}
The enriched composition in $\cc C$ is inherited from the enriched composition in $\M$.
We write $[\cc C,\M]$ for the  category of $\M$-enriched covariant functors from $\cc C$ to $\M$,
and refer to \cite[Section~4]{DRO1} for its projective model structure
--- the weak equivalences and fibrations are defined pointwise.
\vspace{0.1in}

Voevodsky \cite{Voe2} defined the morphisms in $\mathrm{Sm}/k_{+}$ by setting
$$
\mathrm{Sm}/k_{+}(X,Y)
:=
\Hom_{\mathrm{Shv}_{\bullet}(\mathrm{Sm}/k)}(X_{+},Y_{+}),
\quad X,Y\in \mathrm{Sm}/k.
$$
In case $X$ is connected we have $\mathrm{Sm}/k_{+}(X,Y)=\Hom_{\mathrm{Sm}/k}(X,Y)_{+}$ by \cite[Example 2.1]{Voe2}.

\begin{lem}
\label{cfr0}
With the notation above we have identifications of constant simplicial sets
$$
[U_{+},V_{+}](X)
=
\Hom_{\M}((U\times X)_{+},V_{+})
=
\mathrm{Sm}/k_{+}(U\times X,V),
$$
where $U,V,X\in \mathrm{Sm}/k$.
\end{lem}
\begin{proof}
By definition we have
$$
[U_{+},V_{+}](X)
=
\uhom_{\M}(U_{+},V_{+})(X)
=
\Hom_{\M}((U\times X)_{+},V_{+}).
$$
It is evident that $\Hom_{\M}((U\times X)_{+},V_{+})=\mathrm{Sm}/k_{+}(U\times X,V)$.
\end{proof}

\begin{rem}
In fact $\mathrm{Sm}/k_{+}(-,V)$, $V\in \mathrm{Sm}/k$, is the Nisnevich sheaf associated to the presheaf $U\longmapsto\Hom_{\mathrm{Sm}/k}(U,V)\sqcup \pt$.
\end{rem}

Our first example is Segal's category $\Gamma^{\op}$ of pointed finite sets and pointed maps.
\begin{ex}
\label{firstgammaframed}
As in \cite[Section 5]{GP1} we view $\Gamma^{\op}$ as a full subcategory of $\M$ by
sending $K\in\Gamma^{\op}$ to $(\sqcup_{K\smallsetminus\ast} \pt)_{+}$
--- the coproduct is indexed by the non-based elements in $K$.
This turns $\Gamma^{\op}$ into a symmetric monoidal $\M$-category.
Hence $[\Gamma^{\op},\M]$ is a closed symmetric monoidal category by~\cite{Day}.

We claim that $[\Gamma^{\op},\M]$ can be identified with the category $\Gamma\M$ of covariant
functors from $\Gamma^{\op}$ to $\M$ sending $0_{+}$ to the basepoint $*$ of $\cc M$
---
in this case $\cc C=\{\sqcup_{K\smallsetminus *} \pt\mid K\in\Gamma^{\op}\}$:
An $\M$-enriched functor $\cc X\in[\Gamma^{\op},\M]$ sends $K\in\Gamma^{\op}$ to $\cc X(\sqcup_{K\smallsetminus *} \pt^{})\in\cc M$ and for $K,L\in\Gamma^{\op}$ there is a morphism
$$
\alpha_{K,L}
\colon
\cc X(\sqcup_{K\smallsetminus *}\pt)\bigwedge_{\M}[\sqcup_{K\smallsetminus *}\pt,\sqcup_{L\smallsetminus *}\pt]
\longrightarrow
\cc X(\sqcup_{L\smallsetminus *}\pt).
$$
Here the motivic space $[\sqcup_{K\smallsetminus*}\pt,\sqcup_{L\smallsetminus*}\pt]$ is given by
$$
U
\longmapsto
\Gamma^{\op}(\sqcup_{K\times n(U)_{+}\smallsetminus (*,+)}\pt,\sqcup_{L\smallsetminus*}\pt),
$$
where $n(U)$ is the number of connected components of $U\in\mathrm{Sm}/k$ and $n(U)_{+}=\{0,1,\ldots,n{(U)}\}$.
Since $\cc X$ takes values in simplicial sheaves it follows that
$$
\cc X(K)(U)
=
\cc X(K)(U_1)\times\bl{n(U)}\cdots\times\cc X(K)(U_{n{(U)}}),
$$
and consequently we have
$$
\alpha_{K,L}(U)
=
\alpha_{K,L}(U_1)\times\bl{n(U)}\cdots\times\alpha_{K,L}(U_{n{(U)}}).
$$

To a morphism $f:K\longrightarrow L$ in $\Gamma^{\op}$ we associate the morphism $\cc X(\sqcup_{K\smallsetminus *}\pt)\longrightarrow \cc X(\sqcup_{L\smallsetminus *}\pt)$
with $U$-sections
$$
\alpha_{K,L}(U_1)(f)\times\bl{n(U)}\cdots\times\alpha_{K,L}(U_{n{(U)}})(f).
$$
Clearly,
this yields the identification of $[\Gamma^{\op},\M]$ with pointed functors from $\Gamma^{\op}$ to $\M$.
\end{ex}

We are passing to the definition of the category of framed motivic spaces $\M^{\fr}$ and to its natural
enrichment over $\M$.
Let $\mathrm{Fr}_+(k)$ be the category of framed correspondences as in ~\cite[Section~2]{GP1}. Let
$\mathrm{Pre}^{\fr}(k)$ be the category of framed presheaves, that is the category of presheaves of sets
on $\mathrm{Fr}_{+}(k)$.
Let $i: \mathrm{Sm}/k\to\mathrm{Sm}/k_{+}\to\mathrm{Fr}_{+}(k)$ be the composite functor.
Recall from~\cite[Section~2]{GP1} that a framed Nisnevich sheaf on $\mathrm{Sm}/k$ is a framed presheaf
such that its restriction to $\mathrm{Sm}/k$
via the functor $i$
is a Nisnevich sheaf. Let $Shv_\bullet^{\fr}(k)$ denote the category
of pointed framed Nisnevich sheaves.
The morphisms in this category are just morphisms of pointed framed presheaves.
The {\it category of framed motivic spaces\/} $\cc M^{\fr}$ is the category of simplicial objects in
$Shv_\bullet^{\fr}(k)$. There is a canonically induced faithful functor
$\iota:\M^{\fr}\rightarrow\M$ obtained from the composite
$i: \mathrm{Sm}/k\to\mathrm{Sm}/k_{+}\to\mathrm{Fr}_{+}(k)$.

Following~\cite[Section~6]{Voe2} there is a natural pairing
$\mathrm{Sm}/k_{+} \times \mathrm{Fr}_{+}(k)\xrightarrow{\otimes} \mathrm{Fr}_{+}(k)$
taking $(X,Y)$ to $X\times Y$ and $(f,\alpha)$ to $f\times \alpha$. In what follows
this pairing will be used systematically without referring to it.
We also use it in the natural enrichment of $\M^{\fr}$ over $\cc M$.

First, we can associate a framed Nisnevich sheaf $\cc F(X\times -)$ to 
every framed Nisnevich sheaf $\cc F$ and every $X\in \mathrm{Sm}/k_{+}$. In detail,
given $\alpha\in \mathrm{Fr}_n(U',U)$ put
$\alpha^*: \cc F(X\times U)\to \cc F(X\times U')$
to be $(\id_X\times \alpha)^*$.
If $\cc F$ is a pointed framed Nisnevich sheaf then the framed Nisnevich sheaf 
$\cc F(X\times -)$ is pointed also.

Second, every morphism $f: X'\to X$ in $\mathrm{Sm}/k_{+}$ induces a morphism
of framed sheaves $f^*: \cc F(X\times -)\to \cc F(X'\times -)$. Namely, if $U\in \mathrm{Fr}_+(k)$ one sets
$f^*: \cc F(X\times U)\to \cc F(X'\times U)$ to be $(f\times \id_U)^*$.
If $\cc F$ is a pointed framed Nisnevich sheaf, then the morphism of framed sheaves
$f^*: \cc F(X\times -)\to \cc F(X'\times -)$
is a morphism of pointed framed Nisnevich sheaves.

Finally, similarly to~\eqref{tutu}, $\M^{\fr}$ is naturally enriched over $\cc M$. Namely,
$$
{\M}(A,B)(X):=\Hom_{\M^{\fr}}(A,B(X\times\Delta[\bullet]\times-)),
\quad A,B\in\M^{\fr}, \ X\in \mathrm{Sm}/k.
$$
The enriched composition in $\M^{\fr}$ is inherited from the enriched composition in $\cc M$.


Our second example is Voevodsky's category of framed correspondences of level zero.

\begin{ex}
\label{exgammaframed}
We enrich $\mathrm{Sm}/k_{+}$ in $\M$ by setting
$$
[U,V]
:=
\uhom_{\M}(U_{+},V_{+}),
\quad
U,V\in \mathrm{Sm}/k_{+}.
$$
This turns $\mathrm{Sm}/k_{+}$ into a symmetric monoidal $\M$-category with tensor products $U\times V\in\mathrm{Sm}/k$.
It follows that $[\mathrm{Sm}/k_{+},\M]$ is a symmetric monoidal $\M$-category \cite{Day}.
Framed correspondences of level zero form the underlying category of the $\cc M$-category $\mathrm{Sm}/k_{+}$.
According to Lemma \ref{cfr0} the pointed motivic space $[U,V]$ has $Y$-sections the constant simplicial set
$$
[U,V](Y)
=
\Hom_{\cc M}((U\times Y)_{+},V_{+})
=
\mathrm{Sm}/k_{+}(U\times Y,V).
$$
Owing to the $\M$-enrichment every $\cc X\in[\mathrm{Sm}/k_{+},\M]$ gives rise to a morphism
$$
[U,V]
\longrightarrow
\underline{\Hom}_{\cc M}(\cc X(U),\cc X(V)).
$$
On $Y$-sections we obtain a morphism from $[U,V](Y)$ to
\begin{multline*}
\underline{\Hom}_{\cc M}(\cc X(U),\cc X(V))(Y)
=
{\bf S}_{\bullet}(\cc X(U)\wedge Y_{+},\cc X(V))
=\\
{\bf S}_{\bullet}(\cc X(U),\cc X(V)(Y\times-))
=
\Hom_{\cc M}(\cc X(U)\wedge\Delta[\bullet]_+,\cc X(V)(Y\times-)).
\end{multline*}
\end{ex}


The monoidal product $\Gamma^{\op}\boxtimes \mathrm{Sm}/k_{+}$ is the $\M$-category
with objects $\Ob\Gamma^{\op}\times\Ob \mathrm{Sm}/k_{+}$ and
$$
[(K,A),(L,B)]
=
[K,L]
\times
[A,B].
$$
Note that $\Gamma^{\op}\boxtimes \mathrm{Sm}/k_{+}$ is a symmetric monoidal $\M$-category.

\begin{dfn}
\label{dfn:gammasp}
\begin{itemize}
\item[(1)]
A {\it motivic $\Gamma$-space\/} is an $\M$-enriched functor $\cc X\colon \Gamma^{\op}\boxtimes \mathrm{Sm}/k_{+}\longrightarrow\M$.
\item[(2)]
A {\it framed motivic $\Gamma$-space\/} is an $\M$-enriched functor $\cc X:\Gamma^{\op}\boxtimes \mathrm{Sm}/k_{+}\longrightarrow\M^{\fr}$.
\end{itemize}
\end{dfn}

\begin{rem}
Let $\Gamma^{\op}\times\mathrm{Sm}/k_{+}$ denote the underlying category of the $\cc M$-category
$\Gamma^{\op}\boxtimes\mathrm{Sm}/k_{+}$. Every motivic $\Gamma$-space
$\cc X\colon \Gamma^{\op}\boxtimes \mathrm{Sm}/k_{+}\longrightarrow\M$ gives rise to a functor
$\cc X\colon \Gamma^{\op}\times \mathrm{Sm}/k_{+}\longrightarrow\M$ denoted by the same letter.

Unravelling the previous definition, a framed motivic $\Gamma$-space is equivalent to giving the following data:
\begin{itemize}
\item[$\diamond$] an $\cc M$-functor
$\cc X:\Gamma^{\op}\boxtimes \mathrm{Sm}/k_{+}\to\M$;
\item[$\diamond$] a functor
$\cc X':\Gamma^{\op}\times \mathrm{Sm}/k_{+}\to\M^{\fr}$;
\item[$\diamond$]
the induced functor $\cc X:\Gamma^{\op}\times\mathrm{Sm}/k_{+}\to\cc M$
equals the composite functor $\Gamma^{\op}\times\mathrm{Sm}/k_{+}\xrightarrow{\cc X'}\cc M^{\fr}\xrightarrow{\iota}\cc M$
such that the canonical morphism
$$
[U,V](Y)
\longrightarrow
\Hom_{\cc M}(\cc X(K,U),\cc X(K,V)(Y\times-))
$$
factors through $\Hom_{\cc M^{\fr}}(\cc X'(K,U),\cc X'(K,V)(Y\times-))$ for all
$K\in\Gamma^{\op}$, $U,V,Y\in\textrm{Sm}/k_{+}$.
\end{itemize}
\end{rem}

\begin{evaluationfunctors}
Every motivic $\Gamma$-space $\cc X\in [\Gamma^{\op}\boxtimes \mathrm{Sm}/k_{+},\M]$
and $U\in \mathrm{Sm}/k_{+}$ gives rise to an enriched functor $\cc X(U)\in [\Gamma^{\op},\M]$.
In Example~\ref{firstgammaframed} we identified $\cc X(U)$ with the datum of a pointed functor from $\Gamma^{\op}$ to $\M$.
Following~\cite[Example 2.1.2.1]{DGM} by the sphere spectrum we mean the inclusion
$\bb S:\Gamma^{\op}\hookrightarrow {\bf S}_{\bullet}$. By taking the left Kan extension along the sphere spectrum
$\bb S\colon\Gamma^{\op}\hookrightarrow {\bf S}_{\bullet}$ we obtain the
evaluation functor with values in motivic $S^{1}$-spectra
\begin{equation}
\label{equation:S1evaluation}
\ev_{S^{1}}
\colon
[\Gamma^{\op},\M]
\longrightarrow
\Sp_{S^{1}}(k)
\end{equation}
$$
\cc X(U)\longmapsto
\cc X(\bb S,U)
=
(\cc X(S^{0})(U),\cc X(S^{1})(U),\cc X(S^2)(U),\ldots).
$$
We refer to $\cc X(\bb S,\pt)$ as the {\it underlying motivic $S^{1}$-spectrum\/} of $\cc X$.
\vspace{0.1in}

On the other hand,
for $K\in\Gamma^{\op}$ we obtain an enriched functor $\cc X(K)\in [\mathrm{Sm}/k_{+},\M]$
--- see Example~\ref{exgammaframed}.
Moreover,
for $U,V\in \mathrm{Sm}/k_{+}$ there are natural morphisms in $\M$
$$
V_{+}
\longrightarrow
[U,U\times V]\longrightarrow\underline{\Hom}_{\M}(\cc X(K)(U),\cc X(K)(U\times V)).
$$
By adjunction we obtain morphisms
\begin{multline}
\label{equation:mapxauv}
\quad\quad\cc X(K)(U)\wedge V_{+}\longrightarrow\cc X(K)(U\times V)\quad\textrm{and}\\
\cc X(K)(U)\longrightarrow\underline{\Hom}_{\cc M}(V_+,\cc X(K)(U\times V)).\quad\quad
\end{multline}
\vspace{-0.1in}

The simplices of $\gmp\in\Delta^{\op}\mathrm{Sm}/k_+$ consist of finite disjoint unions ${\bb G}^{\sqcup_{<\infty}}_{m}$
of copies of the multiplicative group scheme ${\bb G}_{m}$ and $\pt$. Namely,
the simplices are ${\bb G}_{m}$, ${\bb G}_{m}\sqcup\pt$, ${\bb G}_{m}\sqcup\pt\sqcup\pt$, $\ldots$
(we also refer the reader to~\cite[Notation~8.1]{GP1}).
As a special case of \eqref{equation:mapxauv} we have
\begin{multline}
\label{equation:mapfdu}
\cc X(K)(U)\wedge ({\bb G}^{\sqcup_{<\infty}}_{m})_{+}
\longrightarrow
\cc X(K)(U\times {\bb G}^{\sqcup_{<\infty}}_{m})\quad\textrm{and}\\
\cc X(K)(U)\longrightarrow\underline{\Hom}_{\cc M}(({\bb G}^{\sqcup_{<\infty}}_{m})_{+},
\cc X(K)(U\times {\bb G}^{\sqcup_{<\infty}}_{m})).
\end{multline}
For the smash powers of $\gmp$ we define the morphisms
\begin{multline}
\label{equation:mapgmpn}
\quad\cc X(K)(\gmpn)\wedge\gmp_+\longrightarrow\cc X(K)(\bb G^{\wedge n+1})\quad\textrm{and}\\
\cc X(K)(\gmpn)\longrightarrow\underline{\Hom}_{\cc M}(\gmp_+,\cc X(K)(\bb G^{\wedge n+1}))\quad
\end{multline}
to be the geometric realization of
$$
l
\longmapsto\{\cc X(K)((\bb G^{\wedge n})_l)\wedge(\gmp_+)_l
\longrightarrow
\cc X(K)((\bb G^{\wedge (n+1)})_l)\}
$$
and
$$
l
\longmapsto\{\cc X(K)((\bb G^{\wedge n})_l)
\longrightarrow\underline{\Hom}_{\cc M}
((\gmp_+)_l,\cc X(K)((\bb G^{\wedge (n+1)})_l))\}
$$
obtained from \eqref{equation:mapfdu}.
Due to \eqref{equation:mapgmpn} we obtain the evaluation functor with values in motivic $\gmp$-spectra
\begin{equation}
\label{equation:Gevaluation}
\ev_{\bb G}
\colon
[\mathrm{Sm}/k_{+},\M]
\longrightarrow
\Sp_{\gmp}(k)
\end{equation}
$$
\cc X(K)\longmapsto (\cc X(K)(\pt),\cc X(K)(\gmp),\cc X(K)(\bb G^{\wedge 2}),\ldots).
$$
\vspace{-0.1in}

We refer to \cite[Chapter 3, Section 2.3]{DLORV} for a discussion of the category $\Sp_{S^{1},\gmp}(k)$ of motivic $(S^{1},\gmp)$-bispectra.
Its associated homotopy category is equivalent to $\SH(k)$.
Combining \eqref{equation:S1evaluation} and \eqref{equation:Gevaluation} we obtain the evaluation functor:
\begin{equation}
\label{equation:S1Gevaluation}
\ev_{S^{1},\gmp}
\colon
[\Gamma^{\op}\boxtimes \mathrm{Sm}/k_{+},\M]
\longrightarrow
\Sp_{S^{1},\gmp}(k)
\end{equation}
$$
\cc X
\longmapsto
\cc X_{S^{1},\bb G}
=
\ev_{S^{1},\gmp}(\cc X).
$$
More precisely,
for $i,j\geq 0$ we have
$$
\ev_{S^{1},\gmp}(\cc X)_{i,j}=\cc X(S^{i},\bb G^{\wedge j})\in\M.
$$
The evident structure maps turn $\cc X_{S^{1},\bb G}$ into a motivic $(S^{1},\gmp)$-bispectrum.

In turn, let $[\Gamma^{\op}\boxtimes \mathrm{Sm}/k_{+},\M^{\fr}]$ denote the category of $\cc M$-enriched
functors from $\Gamma^{\op}\boxtimes \mathrm{Sm}/k_{+}$ to $\M^{\fr}$.
Its objects are the framed motivic $\Gamma$-spaces following Definition~\ref{dfn:gammasp}. if
$\cc X$ is a framed motivic $\Gamma$-space then the structure morphisms
   $$\cc X(S^{i},\bb G^{\wedge j})\to\underline{\Hom}(S^1,\cc X(S^{i+1},\bb G^{\wedge j}))$$
   $$\cc X(S^{i},\bb G^{\wedge j})\to\underline{\Hom}(\bb G_+,\cc X(S^{i+1},\bb G^{\wedge j+1}))$$
are morphisms in $\cc M^{\fr}$. Therefore
$\cc X_{S^{1},\bb G}\in \Sp^{\fr}_{S^{1},\gmp}(k)$ is a framed motivic $(S^{1},\bb G)$-bispectrum
in the sense of~\cite[Definition~2.1]{GP5}.
Similarly to~\eqref{equation:S1Gevaluation}
we obtain the evaluation functor:
\begin{equation}
\label{equation:S1Gevaluationfr}
\ev_{S^{1},\gmp}
\colon
[\Gamma^{\op}\boxtimes \mathrm{Sm}/k_{+},\M^{\fr}]
\longrightarrow
\Sp_{S^{1},\gmp}^{\fr}(k)
\end{equation}
$$
\cc X
\longmapsto
\cc X_{S^{1},\bb G}
=
\ev_{S^{1},\gmp}(\cc X).
$$
\end{evaluationfunctors}

\begin{ex}
\label{example:liuhvkbh}
For every $X\in\textrm{Sm}/k$ we can form the motivic $\Gamma$-space with sections
$$
(K,U)
\longmapsto
\mathrm{Sm}/k_+(-,K\otimes (X\times U)).
$$
Its evaluation is the suspension bispectrum $\Sigma^\infty_{S^1}\Sigma^\infty_{\gmp}X_+$ of $X$.
Similarly,
we can form the special framed motivic $\Gamma$-space $\uhom(X,C_{\ast}\mathrm{Fr})$ with sections
$$
(K,U)
\longmapsto
C_{\ast}\mathrm{Fr}(-,K\otimes (X\times U)).
$$
Its underlying motivic $S^{1}$-spectrum $\uhom(X,C_{\ast}\mathrm{Fr})(\bb S,\pt)$ is the framed motive of $X$~\cite{GP1}.

There is a natural morphism of motivic $\Gamma$-spaces
\begin{equation}
\label{equation:kuygrjfv}
\mathrm{Sm}/k_+(-,-\otimes (X\times-))\longrightarrow C_{\ast}\mathrm{Fr}(-,-\otimes (X\times-)).
\end{equation}
By~\cite[Theorem~11.1]{GP1} the evaluation functor in~\eqref{equation:S1Gevaluation} takes the morphism in \eqref{equation:kuygrjfv} to a stable motivic equivalence.
In particular,
the special framed motivic $\Gamma$-space $\uhom(\pt,C_{\ast}\mathrm{Fr})$ is a model for the motivic sphere $\one$.

By linearization we obtain the special framed motivic  $\Gamma$-space $\uhom(X,C_{\ast}\bb Z\mathrm{F})$ with sections
$$
(K,U)
\longmapsto
C_{\ast}\bb Z\mathrm{F}(-,K\otimes (X\times U)).
$$
The underlying motivic $S^{1}$-spectrum $\uhom(X,C_{\ast}\bb Z\mathrm{F})(\bb S,\pt)$ is the linear framed motive of $X$~\cite{GP1}.
\end{ex}

\begin{ex}
\label{example:lukyegygve}
Let $\cc E$ be a motivic symmetric Thom $T$- or $T^{2}$-spectrum with a bounding constant $d\leq 1$ and contractible alternating group action in the sense of~\cite[Section 1]{GN}
--- the main examples are algebraic cobordism $\MGL$~\cite{VoeICM} and the $T^2$-spectra $\MSL$, $\MSp$ in~\cite{PW} (in all of these cases $d=1$).
Under these assumptions there exists a special framed motivic $\Gamma$-space $\uhom(X,C_{\ast}\mathrm{Fr}^{\cc E})$ with sections
$$
(K,U)
\longmapsto
C_{\ast}\mathrm{Fr}^{\cc E}(-,K\otimes (X\times U)).
$$
The evaluation $\ev_{S^{1},\gmp}(\uhom(X,C_{\ast}\mathrm{Fr}^{\cc E}))$ agrees with $\cc E \wedge X_+$ by the proof of~\cite[Theorem 9.13]{GN}.
Moreover,
$\uhom(X,C_{\ast}\mathrm{Fr}^{\cc E})(\bb S,\pt)$ is the $\cc E$-framed motive of $X$ in the sense of \cite[Section~9]{GN}.

Likewise,
we obtain the special framed motivic $\Gamma$-space $\uhom(X,C_{\ast}\bb Z\mathrm{F}^{\cc E})$,
whose underlying motivic $S^{1}$-spectrum is the linear $\cc E$-framed motive of $X$ defined in \cite[Section~9]{GN}.
\end{ex}

\begin{ex}
\label{example:asdefcoljh}
Suppose that ${\bf A}$ is a strict category of Voevodsky correspondences in the sense of~\cite[Definition 2.3]{GG} 
and there exists a functor $\mathrm{Fr}_{+}(k)\longrightarrow{\bf A}$
which is the identity map on objects.
Examples include finite Milnor-Witt correspondences $\wt{Cor}$~\cite{CF},
finite correspondences $Cor$~\cite{Voe1},
and $K_0^{\oplus}$-correspondences~\cite{Wlk}.
We define $C_{\ast}{\bf A}$ to be the very special framed motivic $\Gamma$-space with sections
the Suslin complex of the Nisnevich sheaf ${\bf A}(-,K\otimes U)^{\nis}$
--- sectionwise we have
$$
(K,U)
\longmapsto
C_{\ast}{\bf A}(-,K\otimes U)^{\nis}.
$$
Note that $\underline{\Hom}(X,C_{\ast}{\bf A})(\bb S,\pt)$ is
the ${\bf A}$-motive of $X$ defined in \cite[Section 2]{GG}, where $\underline{\Hom}(X,C_{\ast}{\bf A})$
stands for the very special framed motivic $\Gamma$-space with sections
$(K,U)\mapsto C_{\ast}{\bf A}(-,K\otimes (X\times U))$.
\end{ex}

\begin{rem}
\label{klbogrj}
The motivic $\Gamma$-spaces in Examples \ref{example:liuhvkbh}, \ref{example:lukyegygve}, and \ref{example:asdefcoljh} share the common trait of factoring through the functor
$\otimes\colon \Gamma^{\op}\boxtimes \mathrm{Sm}/k_{+}\longrightarrow\mathrm{Sm}/k_{+}$.
\end{rem}

\section{Special framed motivic $\Gamma$-spaces and infinite motivic loop spaces}
\label{fmg}
Let $\cc E$ be a motivic $(S^{1},\gmp)$-bispectrum.
Using the $n$th weight motivic $S^{1}$-spectrum $\cc E(n)$ of $\cc E$
--- defined by $\cc E(n)_{i}=\cc E_{i,n}$ ---
we write $\cc E=(\cc E(0),\cc E(1),\ldots)$.
For integers $p,n\in\bb Z$ let $\pi^{\bb A^1}_{p,n}\cc E$ be the Nisnevich sheaf on $\textrm{Sm}/k$ associated to the presheaf
$$
U
\longmapsto
\SH(k)(U_{+}\wedge S^{p-n}\wedge\gmpn,\cc E).
$$
Recall that $\cc E$ is connective if $\pi^{\bb A^1}_{p,n}\cc E=0$ for all $p<n$.
Similarly,
a motivic $S^1$-spectrum $\cc E\in \Sp_{S^1}(k)$ is connective if $\pi^{\bb A^1}_{n}\cc E=0$ for all $n<0$.
For a Nisnevich sheaf $F$ of abelian groups on $\mathrm{Sm}/k$,
let $F_{-1}$ denote the Nisnevich sheaf given by $U\longmapsto\ker(1^{\ast}\colon F(U\times\bb G_m)\longrightarrow F(U))$.

\begin{lem}
\label{connective}
A framed motivic $(S^{1},\gmp)$-bispectrum $\cc E=(\cc E(0),\cc E(1),\ldots)$ in the sense of~\cite[Section~2]{GP5} is connective
if and only if $\cc E(n)$ is a connective motivic $S^{1}$-spectrum for every $n\geq 0$.
\end{lem}
\begin{proof}
Without loss of generality we may assume that the underlying motivic bispectrum $\cc E$ is fibrant
(we use here~\cite[Lemma~2.6]{GP5}).
Writing $|-|$ for the absolute value we have $\pi_{p,n}^{\bb A^1}\cc E=\pi_{p-n}^{\nis}\cc E(|n|)$ if $n\leq 0$,
while $\pi_{p,n}^{\bb A^1}\cc E=\pi_{p-n}^{\nis}\Omega_{\bb G^{\wedge n}}\cc E(0)$ if $n>0$.
Here $\pi_{\ast}^{\nis}$ denotes the Nisnevich sheaf associated to $\pi_{\ast}$.
The proof of the sublemma in~\cite[Section~12]{GP1} shows that
$$
\pi_{p-n}^{\nis}\Omega_{\bb G^{\wedge n}}\cc E(0)
=
\pi_{p-n}^{\nis}\cc E(0)_{-n}.
$$
If $\cc E$ is connective then $\pi_{p-n}^{\nis}\cc E(|n|)=0$ for all $n\leq 0$ and $p<n$.
In particular, for all $s>0$ and $n\leq 0$, the sheaf $\pi_{-s}^{\nis}\cc E(|n|)$ is trivial.
The converse implication is evident.
\end{proof}

Recall that $\Sp_{S^{1}}(k)$ is naturally enriched in  $\cc M$
--- see the proof of~\cite[Theorem 6.3]{H}.
In fact,
for $\cc E,\cc F\in\Sp_{S^{1}}(k)$ one defines $\cc M(\cc E,\cc F)$ as the equalizer of the diagram
\begin{equation}
\label{equation:equalizer}
\prod_n\cc M(\cc E_n,\cc F_n)
\Longrightarrow
\prod_n\cc M(\cc E_n,\uhom_{\cc M}(S^{1},\cc F_{n+1})).
\end{equation}
Here we employ the morphism $\cc M(\cc E_n,\cc F_n)\longrightarrow\cc M(\cc E_n,\uhom_{\cc M}(S^{1},\cc F_{n+1}))$ induced by the
adjoint of the structure maps of $\cc F$, and the canonically induced morphism
$$
\cc M(\cc E_{n+1},\cc F_{n+1})
\longrightarrow
\cc M(\cc E_n\wedge S^{1},\cc F_{n+1})
\cong
\cc M(\cc E_n,\uhom_{\cc M}(S^{1},\cc F_{n+1})).
$$

We shall refer to $\Sp_{S^{1}}([\mathrm{Sm}/k_{+},\cc M])$ as the {\it category of spectral functors\/}
---
see~\cite[Section~5]{GP5}.
The objects are $S^{1}$-spectra in the closed symmetric monoidal $\cc M$-category $[\mathrm{Sm}/k_{+},\cc M]$ introduced in Example \ref{exgammaframed}.
Similarly to \eqref{equation:Gevaluation},
see \cite[Section~5, (3)]{GP5},
there exists an evaluation functor
$$
\ev_{\bb G}
\colon
\Sp_{S^{1}}([\mathrm{Sm}/k_{+},\cc M])
\longrightarrow
\Sp_{S^{1},\bb G}(k).$$

We are now ready to prove Theorem~\ref{theorem:firstmain}.

\begin{proof}[Proof of Theorem~\ref{theorem:firstmain}]
For $\cc X\in \Ho_{\Gamma\M}^{\fr}(k)$ and $n\geq 0$,
see Definition \ref{fmgshc},
the geometric realization functor furnishes the associated $\cc M$-enriched functor
$$
\cc X(\gmpn)
\colon=
\vert
l
\longmapsto
\cc X(-,(\bb G^{\wedge n})_l)
\vert
\in
[\Gamma^{\op},\cc M^{\fr}].
$$
Owing to Example \ref{firstgammaframed} this is a pointed functor from $\Gamma^{\op}$ to $\cc M^{\fr}$.
Applying the functor $\ev_{S^{1}}$ in \eqref{equation:S1evaluation} yields the motivic $S^{1}$-spectrum $\ev_{S^{1}}(\cc X(\gmpn))=\cc X(\bb S,\gmpn)$.
By~\cite[Lemma~2.5]{GP5} $\cc X(\bb S,\gmpn)$ is $\bb A^1$-local. Moreover,
$\cc X(\bb S,\gmpn)$ is sectionwise connective because on every section it is the $S^{1}$-spectrum associated to a $\Gamma$-space.
It follows that $\cc X(\bb S,\gmpn)$ is a connective motivic $S^{1}$-spectrum for every $n\geq 0$.
For the evaluation functor $\ev_{S^{1},\gmp}$ in \eqref{equation:S1Gevaluation} we have
$$
\ev_{S^{1},\gmp}(\cc X)(n)
=
\cc X(\bb S,\gmpn).
$$
Combined with Lemma~\ref{connective} we conclude that $\ev_{S^{1},\gmp}(\cc X)\in \SH(k)_{\geq 0}$.
Therefore we obtain the induced evaluation functor
\begin{equation}
\label{equation:evaluation15}
\ev_{S^{1},\gmp}
\colon
\Ho_{\Gamma\M}^{\fr}(k)
\longrightarrow
\SH(k)_{\geq 0}.
\end{equation}
By the construction of $\Ho_{\Gamma\M}^{\fr}(k)$
--- see Definition \ref{fmgshc} ---
the functor $\ev_{S^{1},\gmp}$ in \eqref{equation:evaluation15} is fully faithful.
It remains to show essential surjectivity
--- this is the most interesting part of the proof.
\vspace{0.1in}

Suppose $\cc E$ is a cofibrant and fibrant symmetric motivic $(S^{1},\gmp)$-bispectrum.
Then there exists a framed spectral functor $\cc M^{\cc E}_{\fr}$ in the sense of~\cite[Definition~6.1]{GP5}
such that $\ev_{\bb G}(\cc M_{\fr}^{\cc E})$ is naturally isomorphic to $\cc E$ in $\SH(k)$
--- see~\cite[Section~6]{GP5};
in fact,
$\cc M^{\cc E}_{\fr}$ enables the equivalence between $\SH(k)$ and framed spectral functors in \cite[Theorem 6.3, Definition~6.5]{GP5}.
\vspace{0.1in}

We briefly recall the construction of $\cc M_{\fr}^{\cc E}$ since it is important for the details of this proof.
The motivic spaces $C_{\ast}\mathrm{Fr}(\cc E_{i,j})$ conspire into a motivic $(S^{1},\gmp)$-bispectrum $C_{\ast}\mathrm{Fr}(\cc E)$.
For $n\geq 0$ we let $R^n_{\gmp} C_{\ast}\mathrm{Fr}(\cc E)$ denote $\uhom(\gmpn,C_{\ast}\mathrm{Fr}(\cc E[n]))$,
where $\cc E[n]$ is the $n$th shift of $\cc E$ in the $\gmp$-direction.
In weight $i\geq 0$ we have the motivic $S^1$-spectrum
$$
R^n_{\gmp} C_{\ast}\mathrm{Fr}(\cc E)(i)
=
\uhom(\gmpn,C_{\ast}\mathrm{Fr}(\cc E(n+i))).
$$
There is a canonical morphism of motivic $(S^{1},\gmp)$-bispectra
$$
R^n_{\gmp} C_{\ast}\mathrm{Fr}(\cc E)
\longrightarrow
R^{n+1}_{\gmp} C_{\ast}\mathrm{Fr}(\cc E),
$$
and we set
$$
R^\infty_{\gmp} C_{\ast}\mathrm{Fr}(\cc E)
:=
\colim(C_{\ast}\mathrm{Fr}(\cc E)\longrightarrow R^1_{\gmp} C_{\ast}\mathrm{Fr}(\cc E)\longrightarrow R^2_{\gmp} C_{\ast}\mathrm{Fr}(\cc E)\longrightarrow\cdots).
$$
Owing to \cite[Claim~2, Section~6]{GP5} there are stable motivic equivalences
$$
\cc E
\longrightarrow
C_{\ast}\mathrm{Fr}(\cc E)
\longrightarrow
R^\infty_{\gmp} C_{\ast}\mathrm{Fr}(\cc E).
$$
For $n\geq 0$ we define the spectral functor $\bb GC_{\ast}\mathrm{Fr}^{\cc E}[n]$ sectionwise by
$$
U
\longmapsto
\uhom(\gmpn,C_{\ast}\mathrm{Fr}(\cc E(n)\wedge U_+)).
$$
By construction there is a natural morphism of spectral functors
$$
\bb GC_{\ast}\mathrm{Fr}^{\cc E}[n]\longrightarrow \bb GC_{\ast}\mathrm{Fr}^{\cc E}[n+1],
$$
and we set
$$
\cc M_{\fr}^{\cc E}
:=
\colim(\bb GC_{\ast}\mathrm{Fr}^{\cc E}[0]\longrightarrow\bb GC_{\ast}\mathrm{Fr}^{\cc E}[1]\longrightarrow\cdots).
$$

By~\cite[Lemma~6.6]{GP5} there is a morphism of motivic $(S^{1},\gmp)$-bispectra
\begin{equation}
\label{equation:refdufuhfbhfj}
\ev_{\bb G}(\cc M_{\fr}^{\cc E})\longrightarrow R^\infty_{\gmp}C_{\ast}\mathrm{Fr}(E^c).
\end{equation}
In every weight,
\eqref{equation:refdufuhfbhfj} is a stable local equivalence of motivic $S^1$-spectra due to~\cite[Lemma~6.7]{GP5}.
This implies the zigzag of stable motivic equivalences
$$
\cc E
\longrightarrow
R^\infty_{\gmp} C_{\ast}\mathrm{Fr}(\cc E)
\longleftarrow
\ev_{\bb G}(\cc M_{\fr}^{\cc E}),
$$
and therefore an isomorphism in $\SH(k)$
\begin{equation}
\label{equation:iptgbbgibfgfkkh}
\ev_{\bb G}(\cc M_{\fr}^{\cc E})
\cong
\cc E.
\end{equation}

For $U\in\mathrm{Sm}/k_{+}$ the motivic $S^{1}$-spectrum $\cc M^{\cc E}_{\fr}(U)$ is not necessarily a sectionwise $\Omega$-spectrum.
However,
the said property holds for the framed spectral functor ${\bb M}^{\cc E}_{\fr}$ with sections
$$
U
\longmapsto
\Theta^{\infty}_{S^{1}}\cc M^{\cc E}_{\fr}(U).
$$
Here $\Theta^{\infty}_{S^{1}}$ is the motivic $S^{1}$-stabilization functor defined in \cite[Definition 4.2]{H}.
By construction there is a canonical morphism
\begin{equation}
\label{equation:morphismfsf}
{\cc M}^{\cc E}_{\fr}(U)
\longrightarrow
{\bb M}^{\cc E}_{\fr}(U).
\end{equation}
We note that \eqref{equation:morphismfsf} is a sectionwise stable equivalence of motivic $S^{1}$-spectra.
\vspace{0.1in}

Next we use \eqref{equation:equalizer} to define the motivic $\Gamma$-space $\Gamma\bb M_{\fr}^{\cc E}$ by setting
$$
\Gamma\bb M_{\fr}^{\cc E}(n_{+},U)
:=
\cc M(\bb S^{\times n},\bb M_{\fr}^{\cc E}(U)),\quad n\geq 0,\ U\in\mathrm{Sm}/k.
$$
Here the $S^{1}$-spectrum $\bb S^{\times n}:=\bb S\times\bl{n}\cdots\times\bb S$ is regarded as a constant motivic $S^{1}$-spectrum.
For all $U,V\in\mathrm{Sm}/k_{+}$ and the adjunction $(\ev_{S^1},\Phi)$ between $\Gamma$-spaces and spectra in~\cite[Section 5]{BFgamma} we have
\begin{equation}
\label{kucherov}
\Gamma\bb M_{\fr}^{\cc E}(n_{+},U)(V)
=
\Phi(\bb M_{\fr}^{\cc E}(U)(V))(n_{+})
=
{\bf S}_{\bullet}(\bb S^{\times n},\bb M_{\fr}^{\cc E}(U)(V)).
\end{equation}
This expression determines the values of $\Phi$ at the $S^{1}$-spectrum $\bb M_{\fr}^{\cc E}(U)(V)$.
Moreover,
the counit $\ev_{S^1}\circ\Phi\longrightarrow\id$ induces a morphism of spectral functors
\begin{equation}\label{tikhonov}
\ev_{S^1}(\Gamma\bb M_{\fr}^{\cc E})\longrightarrow\bb M_{\fr}^{\cc E}.
\end{equation}

By construction,
$\Gamma\bb M_{\fr}^{\cc E}$ is a framed motivic $\Gamma$-space in the sense of Definition~\ref{dfn:gammasp}.
Moreover,
in weight $n\geq 0$,
\eqref{equation:morphismfsf} induces a sectionwise stable equivalence of motivic $S^{1}$-spectra
$$
\ev_{\gmp}(\cc M^{\cc E}_{\fr})(n)
\longrightarrow
\ev_{\gmp}(\bb M^{\cc E}_{\fr})(n).
$$
In combination with \eqref{equation:iptgbbgibfgfkkh} we deduce an isomorphism in $\SH(k)$
\begin{equation}
\label{equation:agfdtsc}
\ev_{\gmp}(\bb M^{\cc E}_{\fr})
\cong
\cc E.
\end{equation}
We will show that $\Gamma\bb M_{\fr}^{\cc E}$ satisfies (1)-(4) in Axioms~\ref{axioms:motivicgammaspaces} and also~(5) provided $\cc E\in \SH(k)_{\geq 0}$.
\vspace{0.1in}

Clearly we have $\Gamma\bb M_{\fr}^{\cc E}(0_{+},U)=\ast=\Gamma\bb M_{\fr}^{\cc E}(n_{+},\emptyset)$
for all $U\in\mathrm{Sm}/k_{+}$ and $n\geq 0$.
Moreover,
the canonical sectionwise stable equivalence of cofibrant motivic $S^{1}$-spectra
$$
\bb S\vee\bl n\cdots\vee\bb S
\longrightarrow
\bb S\times\bl n\cdots\times\bb S
$$
induces
--- via \eqref{equation:equalizer} and \eqref{kucherov} ---
the sectionwise equivalence of motivic spaces
$$
\Gamma\bb M_{\fr}^{\cc E}(n_{+},U)
=
\cc M(\bb S^{\times n},\bb M_{\fr}^{\cc E}(U))
\longrightarrow
\cc M(\bb S^{\vee n},\bb M_{\fr}^{\cc E}(U))
\cong
\Gamma\bb M_{\fr}^{\cc E}(1_{+},U)^{\times n}.
$$
This establishes Axiom~(1).
\vspace{0.1in}

For $U\in\mathrm{Sm}/k_{+}$ the presheaf of stable homotopy groups $\pi_{n}\ev_{S^{1}}({\Gamma\bb M}^{\cc E}_{\fr}(U))$
is isomorphic to $\pi_{n}({\bb M}^{\cc E}_{\fr}(U))$ if $n\geq 0$ and trivial if $n<0$
--- this follows as in~\cite[Theorem~5.1]{BFgamma}.
By \eqref{equation:morphismfsf} there is an isomorphism of presheaves between
$\pi_{\ast}({\cc M}^{\cc E}_{\fr}(U))$ and $\pi_{\ast}({\bb M}^{\cc E}_{\fr}(U))$.
Since the former is framed in addition to being $\bb A^1$- and $\sigma$-invariant,
the same holds for $\pi_{\ast}\ev_{S^{1}}({\Gamma\bb M}^{\cc E}_{\fr}(U))$.
This shows that Axiom~(2) holds.
\vspace{0.1in}

Axioms~(3) and~(4) hold because ${\bb M}^{\cc E}_{\fr}$ is a framed spectral functor and the presheaves of stable homotopy groups $\pi_n\ev_{S^{1}}({\Gamma\bb M}^{\cc E}_{\fr}(U))$
of the connective $\bb A^1$-local motivic $S^{1}$-spectrum $\ev_{S^{1}}({\Gamma\bb M}^{\cc E}_{\fr}(U))$
are isomorphic to $\pi_n({\bb M}^{\cc E}_{\fr}(U))$ for all $n\geq 0$
and $U\in\mathrm{Sm}/k_{+}$.
\vspace{0.1in}

Axiom~(5) holds if we assume $\cc E\in \SH(k)_{\geq 0}$.
Indeed,
the proof of~\cite[Theorem 6.3]{GP5} shows $\cc E\wedge U_{+}\in \SH(k)_{\geq 0}$ is isomorphic to $\ev_{\bb G}(\bb M^{\cc E}_{\fr}(-\times U))$ for all $U\in\mathrm{Sm}/k_{+}$.
Here $\bb M^{\cc E}_{\fr}(-\times U)$ is the framed spectral functor with sections
$$
X
\longmapsto
\bb M^{\cc E}_{\fr}(X\times U).
$$
By Lemma~\ref{connective} the $\bb A^1$-local motivic $S^{1}$-spectrum $\bb M^{\cc E}_{\fr}(U)$ is connective.
Indeed,
$\bb M^{\cc E}_{\fr}(U)$ is the zeroth weight of the framed bispectrum $\ev_{\bb G}(\bb M^{\cc E}_{\fr}(-\times U))$ whose weights are $\bb A^1$-local by~\cite[Lemma~2.6]{GP5}.
Thus for all $U\in\mathrm{Sm}/k_{+}$ the morphism~\eqref{tikhonov} yields a stable local equivalence of connective motivic $S^{1}$-spectra
$$
\Gamma\bb M^{\cc E}_{\fr}(\bb S,U)
\longrightarrow
\bb M^{\cc E}_{\fr}(U).
$$
We conclude $\Gamma\bb M^{\cc E}_{\fr}(\bb S,-)$ is a framed spectral functor and the framed motivic $\Gamma$-space $\Gamma\bb M^{\cc E}_{\fr}$ satisfies Nisnevich excision as in Axiom~(5).
This completes the proof.
\end{proof}

\begin{rem}
\label{sergachev}
The proof of Theorem~\ref{theorem:firstmain} shows that a quasi-inverse functor $\Gamma\bb M_{\fr}$ to the equivalence
$\ev_{S^{1},\gmp}:\Ho_{\Gamma\M}^{\fr}(k){\longrightarrow}\SH(k)_{\geq 0}$ is given as follows:
For $\cc E\in\SH(k)_{\geq 0}$ take a functorial cofibrant and fibrant replacement $\cc E^{\prime}$ in the stable model structure on symmetric motivic $(S^{1},\bb G)$-bispectra.
Then map $\cc E$ to the framed motivic $\Gamma$-space $\Gamma\bb M_{\fr}^{\cc E^{\prime}}$.
\end{rem}

With Theorem~\ref{theorem:firstmain} in hand we can prove Theorem~\ref{theorem:secondmain}.

\begin{proof}[Proof of Theorem~\ref{theorem:secondmain}]
Following \cite[Section 3, p.~1131]{Bachmann}, \cite[Section 5]{SO} we have
$$
\SH^{\veff}(k)=\SH(k)_{\geq 0}\cap \SH^{\eff}(k),
$$
where $\SH^{\eff}(k)$ is the full subcategory of $\SH(k)$ comprised of effective bispectra.
For $\cc X\in \Ho_{\Gamma\M}^{\veffr}(k)$ the evaluation $\cc X_{S^{1},\bb G}$ is contained in $\SH(k)_{\geq 0}$ due to Theorem~\ref{theorem:firstmain}.
By Axiom (6) the $S^{1}$-spectrum
$$
|\cc X(\bb S,\gmp\times U)(\wh{\Delta}^\bullet_{K/k})|
$$
is stably contractible for any finitely generated field extension $K/k$ and $U\in\mathrm{Sm}/k$.
It follows that
$$
|\cc X(\bb S,\gmpn)(\wh{\Delta}^\bullet_{K/k})|
$$
is stably contractible for every $n>0$.
We deduce that $\cc X_{S^{1},\bb G}\in\SH^{\eff}(k)$ and thus $\cc X_{S^{1},\bb G}\in\SH^{\veff}(k)$ by~reference to \cite[Theorem 4.4]{BF} and~\cite[Definition 3.5, Theorem 3.6]{GP5}.
\vspace{0.1in}

We have shown the restriction of the equivalence $\ev_{S^{1},\gmp}:\Ho_{\Gamma\M}^{\fr}(k)\lra{\simeq}\SH(k)_{\geq 0}$ in Theorem~\ref{theorem:firstmain} to the full subcategory
$\Ho_{\Gamma\M}^{\veffr}(k)$ takes values in $\SH^{\veff}(k)$.
It remains to show that it is essentially surjective.
\vspace{0.1in}

Suppose $\cc E$ is a very effective cofibrant and fibrant symmetric motivic $(S^{1},\bb G)$-bispectrum.
By Theorem~\ref{theorem:firstmain} there exists a framed motivic $\Gamma$-space $\Gamma\bb M^{\cc E}_{\fr}$ and an isomorphism between $\ev_{S^{1},\gmp}(\Gamma\bb M^{\cc E}_{\fr})$
and $\cc E$ in $\SH(k)_{\geq 0}$.
Moreover,
the proof of Theorem~\ref{theorem:firstmain} shows that for every $U\in\mathrm{Sm}/k_{+}$ there is an isomorphism in $\SH(k)_{\geq 0}$ between $\cc E\wedge U_{+}$ and
$\ev_{S^{1},\gmp}(\Gamma\bb M^{\cc E}_{\fr}(-\times U))$.
Here $\Gamma\bb M^{\cc E}_{\fr}(-\times U)$ is the framed motivic $\Gamma$-space with sections
$$
(n_{+},X)
\longmapsto
\Gamma\bb M^{\cc E}_{\fr}(n_{+},X\times U).
$$
Recall that $\SH^{\veff}(k)$ is closed under the smash product in $\SH(k)$ by \cite[Lemma 5.6]{SO}.
In particular we have $\cc E\wedge U_{+}\in \SH^{\veff}(k)$.
To conclude the $S^{1}$-spectrum
$$
|\Gamma\bb M^{\cc E}_{\fr}(\bb S,\gmp\times U)(\wh{\Delta}^\bullet_{K/k})|
$$
is stably contractible we appeal to~\cite[Theorem~3.6]{GP5}.
It follows that the framed motivic $\Gamma$-space $\Gamma\bb M^{\cc E}_{\fr}$ is effective,
and hence $\cc E$ is isomorphic to $\ev_{S^{1},\gmp}(\Gamma\bb M^{\cc E}_{\fr})$ in $\SH^{\veff}(k)$.
\end{proof}

Suppose $\cc E$ is a motivic $(S^{1},\bb G)$-bispectrum with a motivic fibrant replacement $\cc E^{f}$.
We will write $\Omega_{S^{1}}^{\infty}\Omega_{\bb G}^{\infty} \cc E$ for the pointed motivic space $\cc E^{f}_{0,0}$.

\begin{dfn}
\label{zadorov}
A pointed motivic space $A$ is an {\it infinite motivic loop space\/} if there exists a motivic $(S^{1},\bb G)$-bispectrum $\cc E$
and a local equivalence $A\simeq\Omega_{S^{1}}^{\infty}\Omega_{\bb G}^{\infty} \cc E$.
\end{dfn}

\begin{lem}
\label{ove}
Suppose $\cc X$ is a very special framed motivic $\Gamma$-space.
Then the bispectrum ${\cc X}^{f}_{S^{1},\gmp}$ obtained from ${\cc X}_{S^{1},\gmp}$ by taking levelwise local fibrant replacements is motivically fibrant.
\end{lem}
\begin{proof}
This follows from~\cite[Lemma 2.6]{GP5} since the $S^{1}$-spectrum associated with a very special $\Gamma$-space is an $\Omega$-spectrum after taking levelwise fibrant replacements \cite[Corollary 2.2.1.7]{DGM}.
\end{proof}

The above brings us to the proof of Theorem~\ref{theorem:thirdmain}.

\begin{proof}[Proof of Theorem~\ref{theorem:thirdmain}]
Without loss of generality we may assume $\cc E\in\SH(k)_{\geq 0}$.
Indeed, it follows from~\cite[p.~374]{ALP} that
for any $\cc E$ the connective cover $\tau_{\geq 0}\cc E\longrightarrow\cc E$ yields a sectionwise equivalence
$$
\Omega^\infty_{S^1}\Omega^\infty_{\bb G}(\tau_{\geq 0}\cc E)
\longrightarrow
\Omega^\infty_{S^1}\Omega^\infty_{\bb G}(\cc E).
$$
Now every $\cc E\in \SH(k)_{\geq 0}$ is isomorphic to $\ev_{S^{1},\gmp}(\Gamma\bb M_{\fr}^{\cc E})$ for some special framed motivic $\Gamma$-space $\Gamma\bb M_{\fr}^{\cc E}$
--- see the proof of Theorem~\ref{theorem:firstmain}.
For $n\geq 0$ and $U,V\in\mathrm{Sm}/k_{+}$,
item \eqref{kucherov} yields
$$
\Gamma\bb M_{\fr}^{\cc E}(n_{+},U)(V)
=
\Phi(\bb M_{\fr}^{\cc E}(U)(V))(n_{+})
=
{\bf S}_{\bullet}(\bb S^{\times n},\bb M_{\fr}^{\cc E}(U)(V)).
$$
Here $\bb M_{\fr}^{\cc E}(U)(V)$ is the $\Omega$-spectrum $\Theta^{\infty}_{S^{1}}\cc M_{\fr}^{\cc E}(U)(V)$ introduced in the proof of Theorem~\ref{theorem:firstmain}.
It follows that $\Gamma\bb M_{\fr}^{\cc E}(1_{+},U)(V)$ is the zeroth space ${\bb M}_{\fr}^{\cc E}(U)(V)_0$ of the $\Omega$-spectrum ${\bb M}_{\fr}^{\cc E}(U)(V)$.
Thus $\pi_0\bb M_{\fr}^{\cc E}(U)(V)_0$ is an abelian group,
and $\pi_0^{\nis}\Gamma\bb M_{\fr}^{\cc E}(1_{+},U)$ is a Nisnevich sheaf of abelian groups.
This shows that $\Gamma\bb M_{\fr}^{\cc E}$ is a very special framed motivic $\Gamma$-space --- see Axiom (7).
\vspace{0.1in}

By appeal to Lemma~\ref{ove} the bispectrum $\ev_{S^{1},\gmp}(\Gamma\bb M_{\fr}^{\cc E})^f$ obtained by taking levelwise local fibrant replacements is motivically fibrant.
Hence there exists a sectionwise equivalence of pointed motivic spaces
$$
\Gamma\bb M_{\fr}^{\cc E}(1_{+},\pt)^f
=
\ev_{S^{1},\gmp}(\Gamma\bb M_{\fr}^{\cc E})^f_{0,0}
\simeq
\Omega^{\infty}_{S^{1}}\Omega^{\infty}_{\bb G}\cc E.
$$
We conclude that $\Gamma\bb M_{\fr}^{\cc E}(1_{+},\pt)$ is locally
equivalent to $\Omega^{\infty}_{S^{1}}\Omega^{\infty}_{\bb G}\cc E$.
\vspace{0.1in}

Now suppose $\cc X$ is a very special framed motivic $\Gamma$-space.
By Lemma~\ref{ove}, ${\cc X}^{f}_{S^{1},\gmp}$ is motivically fibrant and we deduce
$$
\cc X(1_{+},\pt)^f
=
\ev_{S^{1},\gmp}(\cc X)^f_{0,0}
\simeq
\Omega^{\infty}_{S^{1}}\Omega^{\infty}_{\bb G}{\cc X}^{f}_{S^{1},\gmp}.
$$
Since $\cc X(1_{+},\pt)$ is locally equivalent to $\cc X(1_{+},\pt)^f$ it follows that $\cc X(1_{+},\pt)$ is an infinite motivic loop space in the sense of Definition~\ref{zadorov}.
\end{proof}

\begin{rem}
\label{gavrikov}
Every special framed motivic $\Gamma$-space $\cc X:\Gamma^{\op}\boxtimes\mathrm{Sm}/k_{+}\longrightarrow\cc M$ has a canonically associated very special framed motivic $\Gamma$-space with sections
$$
(n_{+},U)
\longmapsto
\Omega_{S^{1}}\mathrm{Ex}^{\infty}\cc X(S^{1}\wedge n_{+},U).
$$
In this expression, Kan's fibrant replacement functor $\mathrm{Ex}^{\infty}$ is applied sectionwise in ${\bf S}_{\bullet}$.
\end{rem}

We finish this section by discussing the diagram \eqref{equation:jeytrfgvjc} of adjoint functors from the introduction:
\begin{equation*}
\xymatrix{
\Ho(k)
\ar@/^/[rr]^-{\Sigma^{\infty}_{S^{1},\bb G}} \ar@/_/[dr]_-{C_{\ast}\mathrm{Fr}}
&&
\SH(k)_{\geq 0} \ar@/^/[dl]^-{\Gamma\bb M_{\fr}} \ar@/^/[ll]^-{\Omega^{\infty}_{S^{1},\bb G}}
\\
& \Ho_{\Gamma\M}^{\fr}(k) \ar@/^/[ur] \ar@/_/[ul]
}
\end{equation*}
The functor $u\colon \Ho_{\Gamma\M}^{\fr}(k)\longrightarrow\Ho(k)$ sends a framed motivic $\Gamma$-space $\cc X$ to its underlying motivic space $\cc X(1_{+},\pt)$.
Moreover,
$C_{\ast}\mathrm{Fr}$ sends a motivic space $A$ to $C_{\ast}\mathrm{Fr}(A^c\otimes-)$
--- the projective cofibrant replacement $A^c$ of $A$ is a filtered colimit of simplicial smooth schemes from $\Delta^{\op}\mathrm{Sm}/k_{+}$.
According to~\cite[Section~11]{GP1} $C_{\ast}\mathrm{Fr}$ is a functor from $\Ho(k)$ to $\Ho_{\Gamma\M}^{\fr}(k)$.
\vspace{0.1in}

The composite functor $\ev_{S^{1},\gmp}\circ C_{\ast}\mathrm{Fr}$ is equivalent to $\Sigma^{\infty}_{S^{1},\bb G}$ due to~\cite[Section~11]{GP1}.
Theorem~\ref{theorem:thirdmain} implies that $u\circ\Gamma\bb M_{\fr}$ is equivalent to $\Omega^{\infty}_{S^{1},\bb G}$.
Thus the adjoint pair $(\Sigma^{\infty}_{S^{1},\bb G},\Omega^{\infty}_{S^{1},\bb G})$ is
equivalent to $(\ev_{S^{1},\gmp}\circ C_{\ast}\mathrm{Fr},u\circ\Gamma\bb M_{\fr})$.
Since $(\ev_{S^{1},\gmp},\Gamma\bb M_{\fr})$ is an adjoint equivalence by Theorem~\ref{theorem:firstmain}, it follows that
$(C_{\ast}\mathrm{Fr},u)$ is a pair of adjoint functors.

\begin{cor}
\label{vasilevski}
The diagram of adjoint functors~\eqref{equation:jeytrfgvjc} commutes up to equivalence of functors.
\end{cor}

\section{Further properties of motivic $\Gamma$-spaces}
\label{hdtgtfe}
Let $\cc X:\Gamma^{\op}\boxtimes\mathrm{Sm}/k_{+}\longrightarrow \cc M^{\fr}$ be a framed motivic $\Gamma$-space.
One has an enriched functor
$$
\cc X(1_{+},-):
\mathrm{Sm}/k_{+}
\longrightarrow \cc M^{\fr},
\quad
U
\longmapsto
\cc X(1_{+},U).
$$
For all $U,V\in\mathrm{Sm}/k_{+}$ we have the elementary Nisnevich square:
$$
\xymatrix{
\emptyset\ar[d]\ar[r]&V\ar[d]\\
U\ar[r]&U\sqcup V
}
$$

If $\cc X$ is (very) special in the sense of Axioms \ref{axioms:motivicgammaspaces},
then Axioms~(1) and~(5) imply the stable local equivalence
\begin{equation}
\label{equation:kuywgfkv}
\cc X(\bb S,U)\vee\cc X(\bb S,V)\longrightarrow\cc X(\bb S,U\sqcup V).
\end{equation}
On the other hand,
the sectionwise stable equivalence
$$
\cc X(\bb S,U)\vee\cc X(\bb S,V)\longrightarrow\cc X(\bb S,U)\times\cc X(\bb S,V)
$$
factors as
$$
\cc X(\bb S,U)\vee\cc X(\bb S,V)\longrightarrow\cc X(\bb S,U\sqcup V)\longrightarrow\cc X(\bb S,U)\times\cc X(\bb S,V).
$$
It follows that the rightmost morphism is a local stable equivalence.
This shows the morphism of motivic spaces
$$
\cc X(1_{+},U\sqcup V)\longrightarrow\cc X(1_{+},U)\times\cc X(1_{+},V)
$$
is a local equivalence,
and likewise for
$$
\cc X(1_{+},n_{+}\otimes U)\longrightarrow\cc X(1_{+},U)\times\bl n\cdots\times\cc X(1_{+},U),\quad n\geq 1.
$$
Here we write $n_{+}\otimes U:=U\sqcup\bl n\cdots\sqcup U\in\textrm{Sm}/k_{+}$.
Axiom~(1) ensures that $\cc X(1_{+},0_{+}\otimes U)=\ast$ since by definition $0_{+}\otimes U:=\emptyset$.
Moreover,
if $\cc X$ is very special then the Nisnevich sheaf $\pi^{\nis}_{0}\cc X(1_{+},U)$ takes values in abelian groups due to Axiom~(7).
We record these observations in the next lemma.

\begin{lem}
\label{shesterkin}
For any very special framed motivic $\Gamma$-space $\cc X$ and $U\in\textrm{Sm}/k_{+}$ the functor
$$
n_{+}
\longmapsto
\cc X(1_{+},n_{+}\otimes U)
$$
is locally a very special $\Gamma$-space.
\end{lem}

Let us fix a cofibrant replacement functor $A\longrightarrow A^c$ in the projective motivic model structure on $\cc M$ in the sense of~\cite[Section~3]{Bl}, \cite{DRO2}
--- $A^c$ is a sequential colimit of simplicial schemes in $\Delta^{\op}\mathrm{Sm}/k_{+}$.
For a motivic $\Gamma$-space $\cc X$, we define the functor $\cc X(1_{+},-):\cc M\longrightarrow\cc M$ by setting
$$
\cc X(1_{+},A)
:=
\colim_{(\Delta[n]\times U)_{+}\longrightarrow A}\cc X(1_{+},\Delta[n]_{+}\otimes U),\quad A\in\cc M.
$$
Here we identify the pointed motivic space $A$ with $\colim_{(\Delta[n]\times U)_{+}\longrightarrow A}(\Delta[n]\times U)_+$.
\vspace{0.1in}

A key property of $\Gamma$-spaces says that if $f:K\longrightarrow L$ is an equivalence in $\mathbf{S}_\bullet$,
then so is $F(f):F(K)\longrightarrow F(L)$ for every $F:\Gamma^\op\longrightarrow\mathbf{S}_\bullet$
--- see~\cite[Proposition~4.9]{BFgamma}, \cite[Lemma 2.2.1.3]{DGM}.
The following result is a motivic counterpart of this property.

\begin{thm}
\label{barbashev}
For any very special framed motivic $\Gamma$-space $\cc X$ the functor
$$
\cc X(1_{+},-)
\colon
\cc M \
\longrightarrow
\cc M,
\quad
A
\longmapsto
\cc X(1_{+},A^c)
$$
takes motivic equivalences to local equivalences of motivic spaces.
Hence if $\cc X$ is a special framed motivic $\Gamma$-space,
then the functor
$$
\cc X(\bb S,-)
\colon
\cc M
\longrightarrow
\Sp_{S^{1}}(k),
\quad
A
\longmapsto
\cc X(\bb S,A^c)
$$
takes motivic equivalences to stable local equivalences of
motivic $S^{1}$-spectra.
\end{thm}

Our proof of Theorem~\ref{barbashev} is inspired by Voevodsky's theory of left 
derived radditive functors as in~\cite[Theorem~4.19]{Voe5}
--- the basic notions we will need in this paper are recalled below.
In this context,
we note that the category $\textrm{Sm}/k_{+}$ has finite coproducts.
\vspace{0.1in}

Recall that a morphism $e:A\longrightarrow X$ in a category $\cc C$ is a {\it coprojection\/} if it is isomorphic to the canonical morphism $A\longrightarrow A\sqcup Y$ for some $Y$ \cite[Section~2]{Voe5}.
A morphism $f:A\longrightarrow X$ in $\Delta^{\op}\cc C$ is a {\it termwise coprojection\/} if for all $i\geq 0$ the morphism $f_i:A_i\longrightarrow X_i$ is a coprojection.
As observed in~\cite[Section~2]{Voe5} a morphism $f:B\longrightarrow A$ and an object $X$ conspire into the pushout:
$$
\xymatrix{B\ar[r]^(.4){e_B}\ar[d]& B\sqcup X\ar[d]\\
A\ar[r]^(.4){e_A}& A\sqcup X }
$$
It follows that there exist pushouts for all
pairs of morphisms $(e,f)$ with $e$ a coprojection whenever $\cc C$
is a category with finite coproducts
---
and likewise for pairs of morphisms $(e,f)$ in $\Delta^{\op}\cc C$,
where $e$ is a termwise coprojection.
Following~\cite[Section~2]{Voe5} a square in $\Delta^{\op}\cc C$ is called an {\it elementary pushout square\/} if it is isomorphic to the pushout square for a pair of morphisms $(e,f)$,
where $e$ is a termwise coprojection.
\vspace{0.1in}

If $\cc C$ has finite coproducts, then for any commutative square $Q$ of the form
$$
\xymatrix{
B\ar[r] \ar[d]& Y\ar[d]\\
A\ar[r] & X }
$$
we define the object $K_Q$ by the elementary pushout square:
$$
\xymatrix{
B\sqcup B\ar[r]\ar[d]& B\otimes\Delta[1]\ar[d]\\
A\sqcup Y\ar[r]& K_Q }
$$
There is a canonically induced morphism $p_Q:K_Q\longrightarrow X$.
An important example is the {\it cylinder} $\cyl(f)$ of a morphism $f:X\longrightarrow X'$;
in terms of the construction above,
this is the object associated to the square:
$$
\xymatrix{
X\ar[r]^f\ar[d]& X'\ar[d]\\
X\ar[r]^f&X' }
$$
By~\cite[Lemma 2.9]{Voe5} the natural morphisms $X'\longrightarrow\cyl(f)$ and $\cyl(f)\longrightarrow X'$ are mutually inverse homotopy equivalences.
\vspace{0.1in}

\begin{lem}\label{resppush}
Suppose $\cc X$ is a special framed motivic $\Gamma$-space.
Then $\cc X(\bb S,-)$ takes elementary pushout squares in $\Delta^{\op}\mathrm{Sm}/k_+$ to homotopy pushout squares in the stable local model structure on motivic $S^{1}$-spectra.
\end{lem}
\begin{proof}
Consider the pushout square in $\Delta^{\op}\mathrm{Sm}/k_+$ with horizontal coprojections:
$$
\xymatrix{
B\ar[r]^(.4){e_B}\ar[d]& B\sqcup X\ar[d]\\
A\ar[r]^(.4){e_A}& A\sqcup X
}
$$
The associated square of spectra
$$
\xymatrix{
\cc X(\bb S,B)\ar[r] \ar[d]&\cc X(\bb S, B\sqcup X)\ar[d]\\
\cc X(\bb S,A)\ar[r] & \cc X(\bb S,A\sqcup X)
}
$$
is a homotopy pushout because by \eqref{equation:kuywgfkv} it is stably locally equivalent to the pushout square:
$$
\xymatrix{
\cc X(\bb S,B)\ar[r] \ar[d]&\cc X(\bb S, B)\vee\cc X(\bb S,X)\ar[d]\\
\cc X(\bb S,A)\ar[r] &\cc X(\bb S, A)\vee\cc X(\bb S,X)
}
$$
By definition,
an elementary pushout square is isomorphic to the pushout square of morphisms $(e,f)$,
where $e$ is a termwise coprojection.
It remains to observe that the geometric realization of a simplicial homotopy pushout square of spectra is a homotopy pushout.
\end{proof}

\begin{cor}
\label{resppushcor}
Suppose $\cc X$ is a special framed motivic $\Gamma$-space and
$$
\xymatrix{
C\ar[r]^(.48){e}\ar[d]_f& D\ar[d]\\
C'\ar[r]^(.48){e'}& D'
}
$$
is an elementary pushout square in $\Delta^{\op}\mathrm{Sm}/k_+$ of morphisms $(e,f)$, where $e$ is a termwise coprojection.
If $\cc X(\bb S,e)$ is a stable local equivalence of spectra, then so is $\cc X(\bb S,e')$.
\end{cor}

\begin{proof}[Proof of Theorem~\ref{barbashev}]
Let $Q$ denote an elementary Nisnevich square in $\textrm{Sm}/k$:
$$
\xymatrix{
U'\ar[r]\ar[d]&X'\ar[d]\\
U\ar[r]&X
}
$$
By applying the cylinder construction and forming pushouts in $\cc M$ we obtain the commutative diagram:
$$
\xymatrix{
U'_{+}\ar[r]\ar[d]&\cyl(U'_{+}\longrightarrow X'_{+})\ar[d]\ar[r]&X'_{+}\ar[d]\\
U_{+}\ar[r]&\cyl(U'_{+}\longrightarrow X'_{+})\sqcup_{U'_{+}}U_{+}\ar[r]&X_{+}
}
$$
Note that $U'_{+}\longrightarrow\cyl(U'_{+}\rightarrow X'_{+})$ is a termwise coprojection and a projective cofibration between projective cofibrant objects of $\cc M$.
Thus $s(Q):=\cyl(U'_{+}\longrightarrow X'_{+})\sqcup_{U'_{+}}U_{+}$ is projective cofibrant \cite[Corollary 1.1.11]{Hov} and $U_{+}\longrightarrow s(Q)$ is a termwise coprojection.
Likewise,
applying the cylinder construction to $s(Q)\rightarrow X_{+}$ and setting $t(Q):=\cyl\bigl(s(Q)\longrightarrow X_{+}\bigr)$ we get a projective cofibration
$$
\xymatrix{
\cyl(Q)\colon s(Q)\ar[r] & t(Q).
}
$$
Here $\cyl(Q)$ is a termwise coprojection and a local equivalence in $\cc M$.
\vspace{0.1in}

In the following we let $J_{\mot}=J_{\proj}\cup J_{\nis}\cup J_{\bb A^1}$ where
$$
J_{\proj}
=
\{\Lambda^{r}[n]_{+}\wedge U_{+}
\longrightarrow
\Delta[n]_{+}\wedge U_{+}\mid U\in\mathrm{Sm}/k,n>0,0\leq r\leq n\},
$$
$$
J_{\nis}
=
\{\Delta[n]_{+}\wedge s(Q)\bigsqcup_{\partial\Delta[n]_{+}\wedge s(Q)}\partial\Delta[n]_{+}\wedge t(Q)
\longrightarrow
\Delta[n]_{+}\wedge t(Q)\mid \textrm{$Q$ is an elementary Nisnevich square}\},
$$
$$
J_{\bb A^1}
=
\{\Delta[n]_{+}\wedge U\times\bb A^1_{+}\bigsqcup_{\partial\Delta[n]_{+}\wedge U\times\bb A^1_{+}}
\partial\Delta[n]_{+}\wedge \cyl(U\times\bb A^1_{+}\longrightarrow U_{+})\longrightarrow\Delta[n]_{+}\wedge\cyl(U\times\bb A^1_{+}
\longrightarrow
U_{+})\mid U\in\mathrm{Sm}/k\}.
$$
We note that every map in $J_{\mot}$ is a termwise coprojection.
According to~\cite[Lemma 2.15]{DRO2} a morphism is a fibration with fibrant codomain in the projective motivic model structure if and only if it has the right lifting property with respect to $J_{\mot}$.
\vspace{0.1in}

Arguing as in~\cite[Proposition 4.9]{BFgamma} the functor $\cc X(1_{+},-)$ maps members of $J_{\proj}$ to local equivalences.
We note that $\cc X(1_{+},-)$ preserves naive simplicial homotopies
--- if $A$ is a pointed motivic space then $\cc X(1_{+},\Delta[1]_{+}\otimes A^c)$ is a cylinder object for $\cc X(1_{+},A^c)$.
Axiom~(4) implies there is a canonically induced  local equivalence
$$
\cc X(1_{+},U\times\bb A^1)\longrightarrow\cc X(1_{+},\cyl(U\times\bb A^1\longrightarrow U)).
$$
Axiom~(5) implies the same holds for $\cc X(1_{+},\cyl(Q))$.
\vspace{0.1in}

To show that $\cc X(1_{+},-)$ maps members of $J_{\nis}$ to local equivalences,
let us start with a cofibration of simplicial sets $K\hookrightarrow L$ and the induced commutative diagram:
$$
\xymatrix{
K_{+}\wedge s(Q)\ar[rr]\ar[d]_{a_0}&&L_{+}\wedge s(Q)\ar[d]_{a_1}\ar@/^/[ddr]^{a_2}\\
K_{+}\wedge t(Q)\ar[rr]\ar@/_/[rrrd]&&K_{+}\wedge t(Q) \bigsqcup_{K_{+}\wedge s(Q)}L_{+}\wedge s(Q)\ar@{.>}[dr]^{a_3}\\ &&&L_{+}\wedge t(Q)
}
$$
Applying Lemma \ref{shesterkin} to $a_0=K_+\wedge\cyl(Q)$ implies the induced morphism $\cc X(1_{+},a_0)$ is a local equivalence.
The same applies to $a_2=L_+\wedge\cyl(Q)$ and $\cc X(1_{+},a_2)$.
Since $\cc X$ is very special,
Corollary~\ref{resppushcor} shows $\cc X(1_{+},a_1)$ is a local equivalence.
Thus $\cc X(1_{+},a_3)$ is a local equivalence and our claim for $J_{\nis}$ follows.
Likewise,
$\cc X(1_{+},-)$ maps members of $J_{\bb A^1}$ to local equivalences.
\vspace{0.1in}

So far we have established that $\cc X(1_{+},-)$ takes members of $J_{\mot}$ to local equivalences.
For every motivic equivalence $f:A\longrightarrow B$ the induced morphism $f^c:A^c\longrightarrow B^c$ is also a motivic equivalence.
It remains to show the canonical morphism
$$
\cc X(1_{+},f^c):\cc X(1_{+},A^c)
\longrightarrow
\cc X(1_{+},B^c)
$$
is a local equivalence.
To that end we apply the small object argument \cite[Theorem 2.1.14]{Hov}.
\vspace{0.1in}

To begin we note that all the morphisms in $J_{\mot}$ have finitely presentable (co)domains.
For every pointed motivic space $A\in\cc M$,
let $\alpha:A\longrightarrow \cc L A$ be the transfinite composition of the $\aleph_{0}$-sequence
$$
A=E^0\lra{\alpha_0}E^1\lra{\alpha_1}E^2\lra{\alpha_2}\cdots
$$
constructed as follows:
For $n\geq 0$ we let $S_{n}$ denote the set of all commutative squares
$$
\xymatrix{
C\ar[r]\ar[d]_g&E^n\ar[d]\\
D\ar[r]& {*}
}
$$
where $g\in J_{\mot}$ and form the pushout:
$$
\xymatrix{
\bigsqcup_{S_{n}}C\ar[r]\ar[d]_{\sqcup g}&E^n\ar[d]^{\alpha_n}\\
\bigsqcup_{S_{n}}D\ar[r]&E^{n+1}
}
$$
This construction is plainly functorial in $A$.
By definition,
$\alpha$ is a trivial motivic cofibration in $\cc M$ belonging to $J_{\mot}$-cell \cite[Definition 2.1.9]{Hov}.

We claim the horizontal morphisms in the commutative diagram
$$
\xymatrix{
\cc X (1_+,A^c)\ar[r]\ar[d]_{\cc X(1_{+},f^c)}&\cc X(1_{+},{\cc L}(A^c))\ar[d]^{\cc X(1_{+},{\cc L}(f^c))}\\
\cc X(1_{+},B^c)\ar[r]&\cc X(1_{+},{\cc L}(B^c))
}
$$
are local equivalences:
Corollary~\ref{resppushcor} shows $\cc X(1_{+},-)$ maps the cobase change of a member of $J_{\mot}$ to a local equivalence
--- here we use the assumption that $\cc X$ is very special.
Local equivalences are closed under filtered colimits and $\cc X(1_{+},-)$ preserves filtered colimits,
so the same holds for members of $J_{\mot}$-cell.
Since ${\cc L}(A^c)$ and ${\cc L}(B^c)$ are cofibrant and fibrant,
${\cc L}(f^c)$ is a homotopy equivalence.
As noted above $\cc X(1_{+},-)$ preserves naive simplicial homotopies and therefore $\cc X(1_{+},{\cc L}(f^c))$ is a homotopy equivalence.
Thus $\cc X(1_{+},f^c)$ is a local equivalence.
\end{proof}

Let $M\bb Z$ be the motivic ring spectrum representing integral motivic cohomology in the sense of
Suslin-Voevodsky \cite{VoeICM}.
Up to inversion of the exponential characteristic $e$ of the base field $k$,
the highly structured category of $M\bb Z$-modules is equivalent to Voevodsky's derived category of motives
---
see \cite[Theorem~58]{RO} and also \cite[Theorem~5.8]{HKO}.
A crucial part of the proof shows that for every $U\in\mathrm{Sm}/k$ the natural assembly morphism
$$
M\bb Z\wedge U_+\longrightarrow M\bb Z\circ(-\wedge U_+)
$$
is an isomorphism in $\SH(k)[1/e]$.
For a $\Gamma$-space $F:\Gamma^{\op}\longrightarrow\bf S_\bullet$ the corresponding statement says that the morphism
$$
\ev_{S^1}(F)\wedge K
\longrightarrow
\ev_{S^1}(F(-\wedge K))
$$
is a stable equivalence for every pointed simplicial set $K\in\bf S_\bullet$
--- see~\cite[Lemma~4.1]{BFgamma}.
We show a similar property for special framed motivic $\Gamma$-spaces.
\begin{thm}
\label{stlouis}
Suppose $k$ is an infinite perfect field of exponential characteristic $e$.
Let $U\in\mathrm{Sm}/k$ be such that $U_+$ is strongly dualizable in $\SH(k)$,
e.g.,
$U$ is a smooth projective algebraic variety.
For every special framed motivic $\Gamma$-space $\cc X$ the natural morphism of bispectra
\begin{equation}
\label{equation:kfjegrv}
\ev_{S^1,\gmp}(\cc X)\wedge U_+
=
\ev_{\gmp}(\cc X(\bb S,-))\wedge U_+
\longrightarrow
\ev_{\gmp}(\cc X(\bb S,-\otimes U))
=
\ev_{S^1,\gmp}(\cc X(-\otimes U))
\end{equation}
is a stable motivic equivalence.
Moreover,
for every pointed motivic space $A\in\cc M$ the natural morphism of bispectra
\begin{equation}
\label{equation:aewsas}
\ev_{S^{1},\bb G}(\cc X)\wedge A^c
\longrightarrow
\ev_{S^{1},\bb G}(\cc X(-\otimes A^c))
\end{equation}
is an isomorphism in $\SH(k)[1/e]$.
\end{thm}
\begin{proof}
Without loss of generality we may assume that $\cc X$ is very special
--- see Remark~\ref{gavrikov}.
We view $\cc X(1_+,-)$ as an $\cc M$-enriched functor from $\mathrm{Sm}/k_+$ to $\cc M$.
\vspace{0.1in}

Recall from \S\ref{section:preliminaries} the $\cc M$-category of finitely presentable motivic spaces $f\cc M$.
Via an enriched left Kan extension functor the inclusion of $\cc M$-categories $\iota:\mathrm{Sm}/k_+\hookrightarrow f\cc M$ yields the functor
$$
\Upsilon
\colon
[\mathrm{Sm}/k_+,\cc M]\longrightarrow[f\cc M,\cc M].
$$
By expressing $\cc Y\in[\mathrm{Sm}/k_+,\cc M]$ as a coend
$$
\cc Y
=
\int^{U\in\mathrm{Sm}/k_+}\cc Y(U)\wedge_{\cc M}[U,-],
$$
we obtain
$$
\Upsilon(\cc Y)
=
\int^{U\in\mathrm{Sm}/k_+}\cc Y(U)\wedge_{\cc M}[\iota(U),-].
$$
By construction,
$\Upsilon(\cc Y)(V)=\cc Y(V)$ for all $V\in\Delta^{\op}\mathrm{Sm}/k_+$.
More generally,
$\Upsilon(\cc Y)(A^c)=\cc Y(A^c)$ for every pointed motivic space $A\in\cc M$.
\vspace{0.1in}

Theorem~\ref{barbashev} implies that $\Upsilon(\cc X(1_+,-))$ maps motivic weak
equivalences of projective cofibrant motivic spaces to local equivalences.
Owing to~\cite[Corollary~56]{RO} the $\bb G$-evaluation of the assembly morphism
$$
\Upsilon(\cc X(1_+,-\otimes\bb S))\wedge U_+\longrightarrow \Upsilon(\cc X(1_+,-\otimes\bb S\otimes U))
$$
is a stable motivic equivalence between motivic $(S^{1},\bb G)$-bispectra if $U_+$ is strongly dualizable in $\SH(k)$.
Here $\cc X(1_+,-\otimes\bb S\otimes U)$ is the evaluation at the sphere $\mathbb S$ of the $\Gamma$-space of Lemma~\ref{shesterkin}.
Since $\Upsilon(\cc X(1_+,V))=\cc X(1_+,V)$ for all $V\in\Delta^{\op}\mathrm{Sm}/k_+$ the same
holds for the $\bb G$-evaluation of the morphism
$$
\cc X(1_+,-\otimes\bb S)\wedge U_+
\longrightarrow
\cc X(1_+,-\otimes\bb S\otimes U).
$$
We denote by $\cc X(S^n,-)$, $n>0$, the very special framed motivic $\Gamma$-space with sections
$$
(k_+,U)
\longmapsto
\cc X(S^n\wedge k_+,U).
$$
Replacing $\cc X$ with $\cc X(S^n,-)$,
we deduce the stable motivic equivalence of motivic $(S^1,\bb G)$-bispectra
\begin{equation}
\label{equation:jyervjcg}
\ev_{\gmp}(\cc X(S^n,-\otimes\bb S))\wedge U_+
\longrightarrow \ev_{\gmp}(\cc X(S^n,-\otimes\bb S\otimes U)).
\end{equation}

Combining \eqref{equation:jyervjcg} with \cite[Lemma 4.1]{BFgamma} we
obtain the stable motivic equivalences of motivic $(S^1,S^1,\bb G)$-trispectra
$$
\ev_{\gmp}(\cc X(\bb S,-\otimes\bb S))\wedge U_+\longrightarrow \ev_{\gmp}(\cc X(\bb S,-\otimes\bb S\otimes U))
$$
and
$$
\ev_{\gmp}(\cc X(\bb S,-))\wedge U_+\wedge\bb S\longrightarrow \ev_{\gmp}(\cc X(\bb S,-\otimes U))\wedge\bb S.
$$
For the cofibrant replacements of $\ev_{\gmp}(\cc X(\bb S,-))\wedge U_+$ and $\ev_{\gmp}(\cc X(\bb S,-\otimes U))$ in $\Sp_{S^1,\bb G}(k)$
we find a stable motivic equivalence between cofibrant motivic $(S^1,S^1,\bb G)$-trispectra
$$
(\ev_{\gmp}(\cc X(\bb S,-))\wedge U_+)^c\wedge\bb S
\longrightarrow
\ev_{\gmp}(\cc X(\bb S,-\otimes U))^c\wedge\bb S.
$$
Since $-\wedge S^1$ is a Quillen auto-equivalence on $\Sp_{S^1,\bb G}(k)$ we deduce the stable motivic equivalence
$$
(\ev_{\gmp}(\cc X(\bb S,-))\wedge U_+)^c
\longrightarrow
\ev_{\gmp}(\cc X(\bb S,-\otimes U))^c
$$
between cofibrant motivic $(S^1,\bb G)$-bispectra
--- see also ~\cite[Theorem 5.1]{H}.
Therefore \eqref{equation:kfjegrv} is a stable motivic equivalence.
\vspace{0.1in}

Recall that $U_+$ is strongly dualisable in $\SH(k)[1/e]$ for every $U\in\mathrm{Sm}/k$
---
see~\cite[Appendix~B]{LYZ}.
The previous arguments show that \eqref{equation:kfjegrv} is an $e^{-1}$-stable motivic equivalence
--- note that~\cite[Corollary~56]{RO} concerns the stable motivic model structure on motivic functors,
but it readily extends to the $e^{-1}$-stable model structure.
\vspace{0.1in}

Finally,
when $A\in\cc M$,
recall that $A^c$ is a sequential colimit of simplicial schemes from $\Delta^{\op}\mathrm{Sm}/k_+$.
Since the geometric realization functor preserves $e^{-1}$-stable motivic equivalences
we conclude \eqref{equation:aewsas} is an isomorphism in $\SH(k)[1/e]$.
\end{proof}

\end{document}